\documentclass[12pt,twoside]{amsart}
\usepackage{geometry}
\usepackage{amsmath,amssymb,amsthm,amscd}
\usepackage{mathtools}
\usepackage{mathrsfs}
\usepackage[shortlabels]{enumitem}
\usepackage{chngcntr}

\usepackage{graphicx}
\usepackage{tabularray}
\usepackage{float}
\usepackage{hhline}
\usepackage[all]{xy}
\usepackage{tikz}
\usepackage{tikz-cd}

\usepackage{xcolor}

\newcommand{\xr}[1]{\textcolor{red}{#1}}

\setlist[enumerate]{leftmargin=56pt,labelsep=8pt,itemsep=4pt,label=\upshape{(\theequation.\arabic*)}}

\makeatletter
\renewcommand{\theequation}{%
\thesection.\arabic{equation}}
\@addtoreset{equation}{section}
\makeatother

\newtheorem{thm}{Theorem}[section]
\newtheorem{lem}[thm]{Lemma}
\newtheorem{cor}[thm]{Corollary}

\newtheorem{claim}[thm]{Claim}

\theoremstyle{definition}
\newtheorem{defn}[thm]{Definition}
\newtheorem{rem}[thm]{Remark}
\newtheorem{ex}[thm]{Example}

\newtheorem{step}{Step}
\newtheorem*{ack}{Acknowledgments}

\newcommand{\Z}{\mathbb{Z}}
\newcommand{\Q}{\mathbb{Q}}
\newcommand{\R}{\mathbb{R}}

\mathtoolsset{showonlyrefs}

\newcounter{stepcounter}
\renewcommand{\step}{%
  \refstepcounter{stepcounter}% 
  \medskip\noindent\textbf{Step \arabic{stepcounter}. }
}

\usepackage{hyperref}

\title{Abundance theorem for log surfaces} %kari
\author{Makoto Enokizono, Nao Moriyama}
\address{Graduate School of Mathematical Sciences, University of Tokyo,
3-8-1 Komaba, Meguro-ku, Tokyo 153-8914, Japan}
\email{enokizono@g.ecc.u-tokyo.ac.jp}
\address{Department of Mathematics, Graduate School of Science, Kyoto University, Kyoto 606-8502, Japan}
\email{moriyama.nao.22s@st.kyoto-u.ac.jp}
\date{}
\begin{document}

\begin{abstract}
We establish the abundance theorem for log surfaces without assuming $\mathbb{Q}$-factoriality or log canonicity.
\end{abstract}

\maketitle
\tableofcontents

\section{Introduction}
\subsection{Main theorem}
Let $k$ be an algebraically closed field.
A {\em log surface} $(X, \Delta)$ is a pair consisting of a normal surface $X$ over $k$ and a boundary $\mathbb{R}$-divisor $\Delta$ on $X$ such that $K_X+\Delta$ is $\mathbb{R}$-Cartier. 
In this paper, we establish the abundance theorem for log surfaces,
which provides a generalization of the abundance theorem for $\Q$-factorial log surfaces or log canonical (lc) %log canonical 
surfaces (\cite{Fuj84}, \cite{FoMc92}, \cite{Fuj12}, \cite{Tan14}):

\begin{thm}[Abundance]\label{thm: abundance}
Let  $(X,\Delta)$ be a log surface, and $\pi\colon X\to S$ be a proper surjective morphism of varieties over $k$.
If $K_X+\Delta$ is $\pi$-nef, then it is $\pi$-semiample.
\end{thm}

The abundance conjecture predicts that, for a projective log pair $(X, \Delta)$ with mild singularities, such as klt or lc, nefness of the log canonical divisor $K_X+\Delta$ implies its semiampleness. 
It is one of the central and most challenging problems in the birational geometry of algebraic varieties. 
In dimension at most three, the conjecture is known to hold for lc pairs in characteristic zero (\cite{KMM94}).

For normal surfaces, the abundance theorem has a long history. 
Fujita \cite{Fuj84} established the semiampleness of the positive parts of Zariski decompositions of log canonical divisors on smooth projective log surfaces. 
In particular, his result implies the abundance theorem for lc surfaces after passing to a minimal resolution.
In a different direction, Sakai \cite{Sak87} essentially proved the abundance theorem for $\mathbb{Q}$-Gorenstein Moishezon surfaces.
Subsequently, the abundance theorem was incorporated into the framework of the minimal model program (MMP).
In characteristic zero, Fong and M\textsuperscript{c}Kernan \cite{FoMc92} gave another proof of the abundance theorem for lc surfaces as part of the proof of the three-dimensional abundance theorem.
%These results were subsequently developed in the framework of the log minimal model program.
%In particular, Fong and McKernan \cite{FoMc92} established the abundance theorem for lc log surfaces as part of the proof of the three-dimensional abundance theorem.
Fujino \cite{Fuj12} later developed the minimal model theory for $\mathbb{Q}$-factorial log surfaces and proved the abundance theorem for $\mathbb{Q}$-factorial log surfaces, without assuming that the pair is lc. %log canonical. 
In positive characteristic, Tanaka \cite{Tan14} established the MMP and the abundance theorem for $\Q$-factorial log surfaces and lc surfaces, using Keel's basepoint-free theorem.
More recently, Fujino \cite{Fuj21} and Nakayama \cite{Nak20} proved the abundance theorem for (generalized) MR log canonical surfaces which form a slightly broader class than log canonical surfaces, in the algebraic and Moishezon cases, respectively.

However, for log surfaces without imposing conditions such as $\mathbb{Q}$-factoriality or log canonicity, the MMP does not run in general.
In fact, there are examples where the MMP and the non-vanishing theorem fail (\cite{Mor26pre}, see Example~\ref{ex: Mor26pre}).
Despite this failure, the abundance theorem nevertheless remains valid in this setting. 
Although we also consider proper log surfaces, the abundance statement in Theorem~\ref{thm: abundance} is nontrivial even when $X$ is projective or when $\Delta=0$.

\subsection{The strategy of the proof}
The strategy of the proof of Theorem~\ref{thm: abundance} is as follows:

Using the method of Shokurov's polytope as in {\cite[Section~3]{Bir11}}, we may assume that $\Delta$ is a boundary $\Q$-divisor (Section \ref{sec: R-div}).

First, we prove the non-vanishing theorem under the assumption that $K_{X}+\Delta$ is nef (Theorem~\ref{thm: non-vanishing}).
Let $\rho\colon Y\to X$ be the minimal resolution and let $\Delta_{Y}$ be an effective $\Q$-divisor on $Y$ defined as
$$
K_{Y}+\Delta_{Y}=\rho^{*}(K_X+\Delta).
$$
By Riemann--Roch and the classification of smooth projective surfaces, it suffices to consider the case where $Y$ is an irrational ruled surface.
After taking a suitable minimal model $Y\to Y_0$ (Lemma~\ref{claim: replace blow-down seq.}), we use the $\mathbb{P}^1$-bundle structure $Y_0\to B$ and the $\Q$-Cartierness of $K_X+\Delta$ to show that $K_{Y}+\Delta_{Y}$ is $\Q$-linearly equivalent to an effective $\Q$-divisor.
Since we do not assume that $X$ is $\Q$-factorial or $(X, \Delta)$ is lc, %log canonical, 
$\Delta_{Y}$ may contain an irreducible component with multiplicity greater than one that is horizontal over $B$.
 
Next, we prove the abundance theorem in the case where $\kappa(X, K_{X}+\Delta)=0$ (Theorem~\ref{thm: log CY}).
Following the argument in the proof of {\cite[Theorem~3.34]{Tan14}}, 
we may reduce to the case where $Y$ is an elliptic ruled surface.
Using the description of the effectivity of $K_Y+\Delta_{Y}$ (Lemmas~\ref{lem: non-vanishing 1} and \ref{lem: non-vanishing 2}),
we show that $K_{Y}+\Delta_{Y}$ is $\Q$-linearly equivalent to the nonnegative multiple of the pullback of $-K_{Y_0}$.
The Iitaka dimension of $-K_{Y_0}$ for a $\mathbb{P}^1$-bundle $Y_0$ over an elliptic curve $B$ is well understood (Lemma~\ref{lem: elliptic ruled}).
Thus, if $K_Y+\Delta_Y\not\sim_{\Q} 0$, then we have
$$
\kappa(-K_{Y_0})=\kappa(K_Y+\Delta_{Y})=0.
$$
Using the classification of such $Y_0$ and the $\Q$-Cartierness of $K_X+\Delta$, we then derive a contradiction.

The abundance theorem for the case $\kappa(K_X+\Delta)=1$ was already established in {\cite[Theorem~4.1]{Fuj84}}.
Thus, the abundance theorem in the absolute and non-big case follows.

We next complete the proof in the non-big case by considering the relative setting where $S$ is a curve (Theorem~\ref{lem: abandance non-big}).
We note that the arguments in both \cite{Fuj12} and \cite{Tan14} reduce the relative setting to the absolute one by applying the MMP for $\Q$-factorial or lc log surfaces.
Since this reduction is unavailable in our setting, the relative case must be treated directly.
Our argument for the relative setting, given in Section \ref{sec: non-big case} is inspired by \cite{Mor26}, which treats the corresponding problem in the analytic setting. 

Finally, we treat the big case, handling the absolute and relative settings simultaneously.
We note that the arguments in both \cite{Fuj12} and \cite{Tan14} rely on the MMP for $\Q$-factorial or lc log surfaces.
The argument in \cite{Fuj12} uses a vanishing theorem in characteristic $0$, whereas that in \cite{Tan14} uses Keel's basepoint-free theorem.
Our proof follows the overall strategy of \cite{Fuj12}, but replaces the MMP-based arguments by the vanishing theorems for normal surfaces established in \cite{Eno23} and \cite{Eno24} (see Corollary~\ref{cor: vanishing}), together with a key adjunction result (Lemma~\ref{lem: adjunction}).

\subsection{Notations and convention}
We work over an algebraically closed field $k$ of arbitrary characteristic throughout this paper.
We collect the following terminologies, which we will use in this paper.

\begin{itemize}
    \item
    A {\em variety} over $k$ refers to an irreducible and reduced scheme that is separated and of finite type over $k$.
    A {\em surface} over $k$ means a variety of dimension $2$.
    \item
    An $\mathbb{R}$-divisor $\Delta$ on a normal variety $X$ is called {\em boundary} if any irreducible component of $\Delta$ has multiplicity in $[0,1]$.
    \item
    We denote by $\sim$ (resp.\ $\sim_{\Q}$, $\equiv$) the linear (resp.\ $\Q$-linear, numerical) equivalence of $\mathbb{R}$-divisors.
    \item
    Let $\pi\colon X\to S$ be a proper surjective morphism.
    An $\mathbb{R}$-divisor $D$ on $X$ is called {\em $\pi$-exceptional} if any irreducible component of $D$ maps to a point by $\pi$ (this is not a standard definition).
    \item
    Let $X$ be a normal proper variety and let $D$ be a $\Q$-Cartier $\Q$-divisor on $X$. 
    Then, the {\it Iitaka dimension} of $D$ is denoted by $\kappa(X,D)$ or $\kappa(D)$.
    \item
    We freely use Mumford's intersection theory and some numerical properties (e.g., nef, big, pseudo-effective) of Weil divisors on normal surfaces (see {\cite[Appendix~A]{Eno24}}).
\end{itemize}

\begin{ack}
The first author is supported by JSPS KAKENHI No.25K06926.
The second author is supported by JSPS KAKENHI No.26KJ1528.
The authors would like to express their sincere gratitude to Professor Osamu Fujino for his continuous guidance to the second author, as well as for his fruitful discussions. We also thank Professor Kenta Hashizume, Professor Noboru Nakayama and Professor Hiromu Tanaka for their valuable comments and helpful suggestions.
\end{ack}

%\section{Preliminaries}

%\subsection{Notations}
%Let $k$ be an algebraically closed field of arbitrary characteristic.
%In this paper, `surface' means a(n) (algebraic) variety over $k$ of dimension two.
%\xr{analytic surfaceも含めるか？}
%Let $X$ be a normal surface.
%Let $\pi\colon X\to S$ be a proper surjective morphism of varieties over %$k$.

%An $\R$-divisor $D$ on $X$ is called \textit{$\pi$-exceptional} when $\pi(D)$ is of pure dimension zero.

%A non-zero effective $\R$-divisor $D$ on $X$ is called \textit{negative definite} when $\operatorname{Supp}D$ is negative definite.

%\xr{この辺の用語は榎園先生の論文を引用で済ませた方が良いか．}

%\xr{メモ：
%resolutionは$\rho\colon Y\to X$，
%relative settingの射は$\pi\colon X\to S$に統一？
%}

\section{Non-vanishing}\label{sec: non-vanishing}

In this section, we prove the non-vanishing theorem for log minimal surfaces: %Theorem~\ref{thm: non-vanishing}.
\begin{thm}[Non-vanishing]\label{thm: non-vanishing}
Let $X$ be a normal proper surface over $k$, and $\Delta$ be a boundary $\mathbb{Q}$-divisor on $X$ such that $K_X+\Delta$ is $\mathbb{Q}$-Cartier.
If $K_X+\Delta$ is nef, then there exists an effective $\mathbb{Q}$-divisor $D$ such that $K_X+\Delta\sim_{\mathbb{Q}} D$.
\end{thm}

\begin{rem}
Theorem~\ref{thm: non-vanishing} fails if we drop any of the assumptions that $K_X+\Delta$ is nef, that $K_X+\Delta$ is $\mathbb{Q}$-Cartier, or that $\Delta$ is a boundary (see Example~\ref{ex: Mor26pre}).  
\end{rem}

Let $X$ be a normal proper surface over $k$, and $\Delta$ be a boundary $\mathbb{Q}$-divisor on $X$ such that $K_X+\Delta$ is $\mathbb{Q}$-Cartier and nef.
Let $\rho\colon Y\to X$ be the minimal resolution of $X$.
Note that $Y$ is a smooth projective surface.
Put 
$$
D:=K_Y+\Delta_Y=\rho^*(K_X+\Delta).
$$
To prove Theorem~\ref{thm: non-vanishing}, it suffices to show that $\kappa(Y, D) \ge 0$. %since $\rho$ is birational.
By the negativity lemma {\cite[Lemma~3.39]{KoMo98}}, the $\Q$-divisor $\Delta_Y$ is effective.
Thus, $\kappa(Y, D) \ge \kappa(Y, K_Y)$.
Then the assertion is immediate if $\kappa(Y, K_Y) \ge 0$.
We may therefore assume that $\kappa(Y, K_Y) = -\infty$.
By the classification of smooth projective surfaces, $Y$ is either a rational surface or an irrational ruled surface.

We first consider the case where $Y$ is a rational surface.
Let $m$ be a positive integer such that $mD$ is an integral divisor.
By the Riemann--Roch theorem for surfaces,
\begin{align*}
h^0(Y, m D) + h^0(Y, K_Y - m D) 
&\ge \frac{1}{2} m D \cdot (m D - K_Y) + \chi(\mathcal{O}_Y)\\
&=\frac{1}{2} m D \cdot ((m-1)D+\Delta_Y) + 1.
\end{align*}
Since $D$ is nef and $\Delta_Y$ is effective, we get
$$
h^0(Y, m D) + h^0(Y, K_Y - m D) \ge 1.
$$
If $h^0(Y, K_Y - m D) = 0$ for some such $m$, then $h^{0}(Y, mD)>0$ and hence $\kappa(Y, D) \ge 0$.
Thus, we may assume that $h^0(Y, K_Y - m D) > 0$ for all $m>0$ such that $mD$ is integral.
In particular, $(K_Y-mD)\cdot H \ge 0$ for any nef divisor $H$ on $Y$ and $m\gg 0$.
This implies that $-(D\cdot H) \ge 0$ for every nef divisor $H$.
Thus, $-D$ is pseudo-effective.
Since $D$ is nef and $-D$ is pseudo-effective, it follows that $D \equiv 0$.
Since $Y$ is a rational surface, numerical and $\mathbb{Q}$-linear equivalence coincide for $\Q$-divisors.
Therefore, $D\sim_\Q 0$,
and in particular $\kappa(Y, D) \ge 0$.

Next we assume that $Y$ is an irrational ruled surface.
Applying the $K_Y$-MMP to $Y$, we obtain a sequence of blow-downs
$$
\varphi\colon
Y=Y_N\xrightarrow{\varphi_N}
Y_{N-1}\xrightarrow{\varphi_{N-1}}
\cdots
\xrightarrow{\varphi_2}
Y_1\xrightarrow{\varphi_1}
Y_0,
$$
where each $\varphi_i\colon Y_i\to Y_{i-1}$ is the blow-up at a point $P_i\in Y_{i-1}$.
The resulting surface $Y_0$ has the structure of  a $\mathbb{P}^1$-bundle $p\colon Y_0\to B$ over a smooth curve $B$ of genus $g_B\ge 1$.
For each $i=0,\cdots, N$, let $E_i:=\varphi_i^{-1}(P_i)$ be the exceptional $(-1)$-curve of $\varphi_{i}$.
Let $D_i=K_{Y_i}+\Delta_{Y_i}$ be the push-forward of $D=K_Y+\Delta_Y$ to $Y_i$.
Since $D_i$ is nef, we can write
$$
D_i=\varphi_i^*D_{i-1}-a_iE_i
$$
for some $a_i\in\Q_{\ge 0}$.
Let $\Gamma_i$ be a fiber of the projection $Y_i\to B$ and put
$$
\alpha:=D_i \cdot \Gamma_i\ge 0.
$$
Note that $\alpha$ is independent of both the choice of $i$ and  the choice of the fiber $\Gamma_{i}$.

\begin{lem}\label{claim: replace blow-down seq.}
After suitably replacing the blow-down sequence $\varphi$, we may assume that 
$a_i \le \frac{\alpha}{2}$ for each $i$.
\end{lem}

\begin{proof}
We proceed by descending induction on $i$ starting from $i=N$.
Assume that $a_j\le \frac{\alpha}{2}$ for all $j=i+1, \cdots, N$.
If $i=0$, then there is nothing to prove.
When $i>0$, we will replace $\varphi_i$ so that $a_i \le \frac{\alpha}{2}$.
Since $Y_i$ is obtained from the $\mathbb{P}^1$-bundle $p\colon Y_0\to B$ by blowing up points on fibers, the ruling $p_i \colon Y_i \to B$ has a reducible fiber $\Gamma_i$ that contains at least one $(-1)$-curves.

First suppose that there exists a $(-1)$-curve $E$ contained in $\Gamma_i$ with the multiplicity at least two.
Since $D_i$ is nef and $2E\le \Gamma_{i}$, we have
$$
D_i\cdot E\le \frac{D_i\cdot \Gamma_i}{2}=\frac{\alpha}{2}.
$$
Thus, replacing $\varphi_i$ with the contraction of $E$, we have $a_i=D_i\cdot E\le\frac{\alpha}{2}$.

We may therefore assume that every $(-1)$-curve contained in $\Gamma_{i}$ has multiplicity one.
In this case, $\Gamma_i$ contains at least two $(-1)$-curves.
Thus, there exists a $(-1)$-curve $E$ in $\Gamma_i$ satisfying 
$$
D_i\cdot E\le \frac{D_i\cdot \Gamma_i}{2}=\frac{\alpha}{2}.
$$
Replacing $\varphi_i$ with the contraction of $E$, we get $a_i=D_i\cdot E\le\frac{\alpha}{2}$.
\end{proof}

From now on, we suppose that $a_i \le \frac{\alpha}{2}$.
To prove Theorem~\ref{thm: non-vanishing}, we propose the following two lemmas:

\begin{lem} \label{lem: non-vanishing 1}
The following holds.
$$
\frac{\alpha}{\alpha+2}\Delta_Y\le \frac{\alpha}{\alpha+2} \varphi^*\Delta_{Y_0}-\sum_{i=1}^N a_i \mathbb{E}_i,
$$
where the divisor $\mathbb{E}_i$ is the total transform of $E_i$ on $Y$.
\end{lem}

\begin{proof}
Since $a_i \le \frac{\alpha}{2}$, we have
\begin{align*}
&\frac{\alpha}{\alpha+2}\varphi^*\Delta_{Y_0}-\sum_{i=1}^N a_i \mathbb{E}_i-\frac{\alpha}{\alpha+2}\Delta_Y\\
&=\frac{\alpha}{\alpha+2}((K_Y-\varphi^{*}K_{Y_0})-(D-\varphi^*D_0))-\sum_{i=1}^N a_i \mathbb{E}_i\\
&=\frac{\alpha}{\alpha+2}\sum_{i=1}^N(a_i+1)\mathbb{E}_i-\sum_{i=1}^N a_i \mathbb{E}_i\\
&=\sum_{i=1}^N\frac{\alpha-2a_i}{\alpha+2}\mathbb{E}_i\\
&\ge 0.
\end{align*}
\end{proof}

\begin{lem} \label{lem: non-vanishing 2}
$\kappa(Y_0, K_{Y_0}+\frac{2}{\alpha+2}\Delta_{Y_0})\ge 0$.
\end{lem}

\begin{proof}
If there are no sections of $p$ contained in $\Delta_{Y_0}$ the proper transform of which is $\rho$-exceptional, 
the claim follows from {\cite[Theorem~2.2]{Fuj84}}.
Indeed, in this case, the horizontal part $\frac{2}{\alpha+2}(\Delta_{Y_0})_{\mathrm{hor}}$ of $\frac{2}{\alpha+2}\Delta_{Y_0}$ is boundary since so is $\Delta$.
Hence {\cite[Theorem~2.2]{Fuj84}} can be applied to $K_{Y_0}+\frac{2}{\alpha+2}(\Delta_{Y_0})_{\mathrm{hor}}$.
Thus, we may assume that there exists a section of $p$ in $\Delta_{Y_0}$ the proper transform of which is $\rho$-exceptional.

Now we use the following terminology.
We fix a rank-two vector bundle $\mathcal{E}$ with $Y_0\cong \mathbb{P}_{B}(\mathcal{E})$ that is normalized in the sense of {\cite[V, Notation~2.8.1]{Har77}}.
Let $T_{\mathcal{E}}$ be the minimal section satisfying $\mathcal{O}_{Y_0}(T_{\mathcal{E}})\cong \mathcal{O}_{\mathbb{P}_{B}(\mathcal{E})}(1)$.
Let $\mathfrak{e}$ be a divisor on $B$ such that $\mathcal{O}_{B}(\mathfrak{e})\cong \det \mathcal{E}$.
Then, we have
$$
K_{Y_0} \sim -2T_\mathcal{E}+p^*(K_B+\mathfrak{e}).
$$
Let $\mathfrak{d}$ be a $\mathbb{Q}$-divisor on $B$ such that
$$
\Delta_{Y_0} \sim_\Q (\alpha+2)T_\mathcal{E}+p^*\mathfrak{d}.
$$
Then we can write
$$
K_{Y_0}+\frac{2}{\alpha+2}\Delta_{Y_0}\sim_\Q p^*(\mathfrak{e}+K_B+\frac{2}{\alpha+2}\mathfrak{d}).
$$
We put
$$
e:=-\deg \mathcal{E}=-\deg \mathfrak{e},\quad d:=\deg \mathfrak{d}.
$$
To prove the claim, we may assume
\begin{equation} \label{eq: condition 1}
-e+2g_B-2+\frac{2}{\alpha+2}d=\deg(\mathfrak{e}+K_B+\frac{2}{\alpha+2}\mathfrak{d})\le 0
\end{equation}
since otherwise the claim trivially follows.
Since $K_{Y_0}+\Delta_{Y_0}$ is nef, we have
\begin{equation}\label{eq: condition 3}
0\le (K_{Y_0}+\Delta_{Y_0})\cdot T_{\mathcal{E}}=-(\alpha+1)e+2g_B-2+d.
\end{equation}

First we assume that $\alpha=0$.
Thus, it follows from \eqref{eq: condition 1} and \eqref{eq: condition 3} that $-e+2g_B-2+d=0$, that is, 
 $K_{Y_0}+\Delta_{Y_0}\sim_{\Q}p^{*}\mathfrak{d}_0$ for some $\Q$-divisor $\mathfrak{d}_0$ on $B$  of degree $0$.
Since $a_i=0$, the $\Q$-divisor class $K_Y+\Delta_{Y}$ is also the pullback of $\mathfrak{d}_0$ to $Y$.
We take a $\rho$-exceptional curve $C$ contained in $\Delta_{Y}$ that is dominant onto $B$.
Since $K_X+\Delta$ is $\Q$-Cartier and $K_Y+\Delta_{Y}=\rho^{*}(K_X+\Delta)$, the restriction $(K_Y+\Delta_{Y})|_{C}\sim_{\Q}\mathfrak{d}_0|_{C}$ is torsion.
Thus $\mathfrak{d}_0$ is also torsion by {\cite[Lemma 2.3]{Fuj84}} and the claim holds.

Thus we may assume that $\alpha>0$.
Since $K_{Y_0}+\Delta_{Y_0}$ is nef, we have
$$
0\le (K_{Y_0}+\Delta_{Y_0})^{2}=-\alpha^{2}e+2\alpha(-e+2g_B-2+d).
$$
Dividing this by $2\alpha$, we have
\begin{equation*} 
-\frac{\alpha+2}{2}e+2g_B-2+d\ge 0.
\end{equation*}
Combining this with \eqref{eq: condition 1} implies 
\begin{align*}
0 &\ge -e+2g_B-2+\frac{2}{\alpha+2}d \\
&\ge -e+2g_B-2+\frac{2}{\alpha+2}\left(\frac{\alpha+2}{2}e-2g_B+2\right) \\
&=\frac{\alpha}{\alpha+2}(2g_B-2) \\
&\ge 0,
\end{align*}
where the last inequality follows from $g_B\ge 1$.
Thus we conclude that
$$
g_B=1,\quad d=\frac{\alpha+2}{2}e, \quad (K_{Y_0}+\Delta_{Y_0})^{2}=0
$$ 
and $K_{Y_{0}}+\frac{2}{\alpha+2}\Delta_{Y_0}\sim_{\Q}p^{*}\mathfrak{d}_0$ for some $\Q$-divisor $\mathfrak{d}_0$ on $B$ of degree $0$.
Moreover, \eqref{eq: condition 3} and $e\ge -1$ ({\cite[V, Theorems 2.12, 2.15]{Har77}}) imply that $e=0$ or $-1$.

Suppose that $e=0$ and $\mathcal{E}$ is decomposable, so that $\mathcal{E}\cong \mathcal{O}_{B}\oplus \mathcal{O}_{B}(\mathfrak{e})$.
Since $d=0$ and $\Delta_{Y_0}$ is effective, we have $\mathfrak{d}\sim_{\Q}a\mathfrak{e}$ for some $a\in \Q$.
Indeed, for a divisor $\mathfrak{b}$ on $B$ and a positive integer $m$, we have
\begin{align*}
H^{0}(Y_0, mT_{\mathcal{E}}+p^{*}\mathfrak{b})&\cong H^{0}(B, \mathrm{Sym}^{m} \mathcal{E}\otimes \mathcal{O}_{B}(\mathfrak{b})) \\
&\cong 
 H^{0}(B, \mathcal{O}_{B}(\mathfrak{b}))\oplus H^{0}(B, \mathcal{O}_{B}(\mathfrak{b}+\mathfrak{e}))\oplus \cdots \oplus H^{0}(B, \mathcal{O}_{B}(\mathfrak{b}+m \mathfrak{e})).
\end{align*}
Thus, if $\deg \mathfrak{b}\le 0$, the linear system $|mT_{\mathcal{E}}+p^{*}\mathfrak{b}|$ is non-empty if and only if $\mathfrak{b}+ a\mathfrak{e}\sim 0$ for some $0\le a\le m$.
This proves $\mathfrak{d}\sim_{\Q}a\mathfrak{e}$ for some $a\in \Q$.
It follows that $K_{Y_{0}}+\frac{2}{\alpha+2}\Delta_{Y_0}\sim_{\Q}\left(1+\frac{2a}{\alpha+2}\right)p^{*}\mathfrak{e}$.
If $\mathfrak{e}$ is torsion, then the right-hand side is $\Q$-linearly trivial, and the desired conclusion holds.
We may therefore assume, for a contradiction, that $\mathfrak{e}$ is not torsion.
Let $T_1=T_{\mathcal{E}}$ and $T_2\sim T_{\mathcal{E}}-p^{*}\mathfrak{e}$ respectively denote the sections of $p$ corresponding to the quotient maps $\mathcal{E}\to \mathcal{O}_{B}(\mathfrak{e})$ and $\mathcal{E}\to \mathcal{O}_{B}$.
These are are the only integral curves on $Y_0$ whose numerical classes are proportional to $T_{\mathcal{E}}$.
Thus, we can write $\Delta_{Y_0}=b_1T_1+b_2T_2$ for some $b_1, b_2\in \Q_{\ge 0}$ with $b_1+b_2=\alpha+2$.
Replacing a normalized vector bundle $\mathcal{E}$  with  $\mathcal{E}\otimes \mathcal{O}_{B}(-\mathfrak{e})$ if necessary, we may assume that $b_1\ge b_2$.
In particular, $b_1\ge (\alpha+2)/2>1$.
Since $\Delta$ is a boundary divisor, it follows that the proper transform $\widetilde{T}_1$ of $T_1$ on $Y$ is $\rho$-exceptional.
Thus, the restriction $(K_Y+\Delta_{Y})|_{\widetilde{T}_1}$ is $\Q$-linearly trivial.
On the other hand, since $K_Y+\Delta_{Y}=\varphi^{*}(K_{Y_0}+\Delta_{Y_0})$ and $K_{Y_0}+\Delta_{Y_0}\sim_{\Q}\alpha T_{\mathcal{E}}+(1-b_2)p^{*}\mathfrak{e}$,
we conclude that the restriction 
$$
(\alpha T_{\mathcal{E}}+(1-b_2)p^{*}\mathfrak{e})|_{T_{\mathcal{E}}}\sim_{\Q}(b_1-1)p^{*}\mathfrak{e}|_{T_{\mathcal{E}}}
$$
 is $\Q$-linearly trivial, where we use $T_{\mathcal{E}}|_{T_{\mathcal{E}}}\sim p^{*}\mathfrak{e}|_{T_{\mathcal{E}}}$ and $b_1+b_2=\alpha+2$.
Since $b_1>1$, we have $\mathfrak{e}\sim_{\Q} 0$ that contradicts the assumption.

Suppose that $e=0$ and $\mathcal{E}$ is indecomposable.
By Atiyah's classification, $\mathcal{E}$ is a non-trivial extension of $\mathcal{O}_{B}$ by $\mathcal{O}_{B}$ and $\mathfrak{e}\sim 0$.
Since $d=0$ and $\Delta_{Y_0}$ is effective, we have $\mathfrak{d}\sim_{\Q}0$.
Indeed, let $C$ be an integral curve contained in the linear system $|mT_{\mathcal{E}}+p^{*}\mathfrak{b}|$,
 where $m>0$ and $\mathfrak{b}$ is a divisor on $B$ with $\deg \mathfrak{b}\le 0$.
 Then either $C=T_{\mathcal{E}}$ or $C\cap T_{\mathcal{E}}=\emptyset$ (the latter possibility can occur only in positive characteristic).
In the latter case, restricting to $T_{\mathcal{E}}$ gives $\mathfrak{b}\sim C|_{T_{\mathcal{E}}}=0$.
Thus, in either case, each component of $\Delta_{Y_0}$ is $\Q$-linearly proportional to $T_{\mathcal{E}}$.
Consequently, $K_{Y_{0}}+\frac{2}{\alpha+2}\Delta_{Y_0}\sim_{\Q}0$.

Suppose that $e=-1$.
In this case, $\mathcal{E}$ is indecomposable, $d=-(\alpha+2)/2$ and 
$$
\Delta_{Y_0}\equiv \frac{\alpha+2}{2}(2T_{\mathcal{E}}-\Gamma)
$$
where $\Gamma$ is the numerical equivalence class of a fiber of $p$.
Note that $(2T_{\mathcal{E}}-\Gamma)^2=0$.
Since $\Delta_{Y_0}$ is numerically proportional to $2T_{\mathcal{E}}-\Gamma$, it follows from {\cite[V, Proposition~2.21]{Har77}} that each irreducible component of $\Delta_{Y_0}$ is also numerically proportional to $2T_{\mathcal{E}}-\Gamma$.
In particular, $\Delta_{Y_0}$ has no sections of $p$.
It is a contradiction since we assume that $\Delta_{Y_0}$ has at least one sections.
\end{proof}

\begin{proof}[Proof of Theorem \ref{thm: non-vanishing}]
By Lemma~\ref{lem: non-vanishing 1}, we have
\begin{align*}
K_Y+\Delta_Y &=\varphi^*(K_{Y_0}+\Delta_{Y_0})-\sum_{i=1}^N a_i \mathbb{E}_i\\
&=\varphi^*\left(K_{Y_0}+\frac{2}{\alpha+2}\Delta_{Y_0}\right)+\varphi^*\left(\frac{\alpha}{\alpha+2}\Delta_{Y_0}\right)-\sum_{i=1}^N a_i \mathbb{E}_i\\
&\ge \varphi^*\left(K_{Y_0}+\frac{2}{\alpha+2}\Delta_{Y_0}\right)+\frac{\alpha}{\alpha+2}\Delta_Y\\
&\ge \varphi^*\left(K_{Y_0}+\frac{2}{\alpha+2}\Delta_{Y_0}\right).
\end{align*}
Thus the claim follows from Lemma~\ref{lem: non-vanishing 2}.
\end{proof}

\section{Log Calabi--Yau case}\label{sec: dim=0 and kappa=0}
%{\texorpdfstring{The absolute case with $\kappa(X, K_X + \Delta) = 0$}{The absolute case with log Kodaira dimension 0}}\label{sec: dim=0 and kappa=0}
In this section, we prove the abundance theorem in the case where $\kappa(X, K_{X}+\Delta)=0$:

\begin{thm} \label{thm: log CY}
Let $X$ be a normal proper surface over $k$, and let $\Delta$ be a boundary $\Q$-divisor on $X$
such that $K_X+\Delta$ is $\Q$-Cartier.
Assume that $K_{X}+\Delta$ is  nef and $\kappa(X, K_{X}+\Delta)=0$.
Then, $K_{X}+\Delta\sim_{\Q}0$.
\end{thm}

Before proving this, we prepare two lemmas:
\begin{lem} \label{lem: elliptic ruled}
Let $p\colon \mathbb{P}_B(\mathcal{E})\to B$ be a geometrically ruled
surface over an elliptic curve $B$,
where $\mathcal{E}$ is normalized in the sense of {\cite[V, Notation~2.8.1]{Har77}}, $\det \mathcal{E}=\mathcal{O}_{B}(\mathfrak{e})$ and $e:=-\deg \mathcal{E}$.
Then
\[
\kappa\left(\mathbb{P}_B(\mathcal{E}), -K_{\mathbb{P}_B(\mathcal{E})}\right)=
\begin{cases}
2, & e>0,\\
1, & e=-1,\\
1, & e=0\text{ and $\mathcal{E}$ is decomposable with
      $\mathfrak e$ torsion}, \\
1, & e=0\text{ and $\mathcal{E}$ is indecomposable in characteristic $p>0$,} \\
0, & e=0\text{ and $\mathcal{E}$ is decomposable with
      $\mathfrak e$ non-torsion}, \\
       0, & e=0\text{ and $\mathcal{E}$ is indecomposable in characteristic $0$.}
\end{cases}
\]
\end{lem}

\begin{proof}
The decomposable case follows directly from the decomposition of $\mathcal{E}\cong \mathcal{O}_{B}\oplus \mathcal{O}_{B}(\mathfrak{e})$ and the formula $-K_{\mathbb{P}_B(\mathcal{E})}\sim 2T_{\mathcal{E}}-p^{*}\mathfrak{e}$.
The case $e=-1$ follows from {\cite[Proposition~3.2]{GP96}}.
It remains to consider the case where $\mathcal{E}$ is indecomposable and $e=0$.
In this case, $\mathcal{E}$ is the non-trivial extension
$$
0\to \mathcal{O}_B\to \mathcal{E}\to \mathcal{O}_B\to 0
$$
and $-K_{\mathbb{P}_B(\mathcal{E})}\sim 2T_{\mathcal{E}}$.
In characteristic $0$, one has $h^{0}(B, \operatorname{Sym}^{m}\mathcal{E})=1$ for every $m\ge 0$ by {\cite[Lemma~17 and Theorem~9]{Ati57}},
and hence $\kappa(-K_{\mathbb{P}_B(\mathcal{E})})=0$.
In positive characteristic, the assertion follows, for example, from {\cite[Proposition~3.32]{Tan14}}.
\end{proof}

\begin{lem}\label{lem: 1st order thickening}
Let $D_1 \subseteq D_2$ be a first order thickening of schemes (i.e., a closed immersion defined by an ideal sheaf $\mathcal{I}_{D_1/D_2}$ such that $\mathcal{I}_{D_1/D_2}^2 = 0$). 
The short exact sequence of sheaves
\[
0 \longrightarrow \mathcal{I}_{D_1/D_2} \xrightarrow{\ \exp\ } \mathcal{O}_{D_2}^{\times} \longrightarrow \mathcal{O}_{D_1}^{\times} \longrightarrow 1,
\]
where $\exp(f) := 1 + f$, induces a long exact sequence in cohomology:
\[
 H^1(D_1, \mathcal{I}_{D_1/D_2}) \xrightarrow{\ H^1(\exp)\ } \operatorname{Pic}(D_2) \xrightarrow{\ \mathrm{res}\ } \operatorname{Pic}(D_1) \longrightarrow 1.
\]
Let $\mathcal{L} \in \operatorname{Pic}(D_2)$ be a line bundle on $D_2$ such that $\mathcal{L}|_{D_1} \cong \mathcal{O}_{D_1}$.
Then, the following hold:

\begin{itemize}
    \item[$(1)$] 
    There exists a short exact sequence of $\mathcal{O}_{D_2}$-modules:
    \begin{align}\label{eq: s.e.s. over D2}
    0 \longrightarrow \mathcal{I}_{D_1/D_2} \longrightarrow \mathcal{L} \longrightarrow \mathcal{O}_{D_1} \longrightarrow 0. 
    \end{align}
    \item[$(2)$] 
    Let $\xi \in H^1(D_1, \mathcal{I}_{D_1/D_2})$ be the extension class of \eqref{eq: s.e.s. over D2}, i.e., $\xi = \delta(1)$ via the connecting homomorphism $\delta \colon H^0(D_1, \mathcal{O}_{D_1}) \to H^1(D_1, \mathcal{I}_{D_1/D_2})$. Then
   \[
    \mathcal{L} = H^1(\exp)(\xi) \in \operatorname{Pic}(D_2).
    \]
\end{itemize}
\end{lem}

\begin{proof}%[Proof of Lemma \ref{lem: 1st order thickening}]
(1) Consider the defining short exact sequence for the ideal sheaf $\mathcal{I}_{D_1/D_2} \subset \mathcal{O}_{D_2}$:
\[
0 \longrightarrow \mathcal{I}_{D_1/D_2} \longrightarrow \mathcal{O}_{D_2} \longrightarrow \mathcal{O}_{D_1} \longrightarrow 0.
\]
Since $\mathcal{L}$ is a locally free $\mathcal{O}_{D_2}$-module of rank one, tensoring this sequence with $\mathcal{L}$ yields the exact sequence
\[
0 \longrightarrow \mathcal{I}_{D_1/D_2} \otimes_{\mathcal{O}_{D_2}} \mathcal{L} \longrightarrow \mathcal{L} \longrightarrow \mathcal{O}_{D_1} \otimes_{\mathcal{O}_{D_2}} \mathcal{L} \longrightarrow 0.
\]
By assumption, $\mathcal{L}|_{D_1} \cong \mathcal{O}_{D_1}$, so $\mathcal{O}_{D_1} \otimes_{\mathcal{O}_{D_2}} \mathcal{L }\cong \mathcal{O}_{D_1}$.
%%%%%%%%
Furthermore, let $\{U_\lambda\}_{\lambda \in \Lambda}$ be an affine open cover of $D_2$. 
For each $\lambda \in \Lambda$, choose a local basis $e_\lambda \in \Gamma(U_\lambda, \mathcal{L})$ of $\mathcal{L}|_{U_\lambda}$ as a free $\mathcal{O}_{D_2}|_{U_\lambda}$-module of rank one. 
The transition functions $g_{\lambda\mu} \in \Gamma(U_\lambda \cap U_\mu, \mathcal{O}_{D_2}^\times)$ representing $\mathcal{L}$ are defined on each overlap $U_\lambda \cap U_\mu$ by the relation
\[
e_\mu = g_{\lambda\mu}  e_\lambda \quad \text{in } \Gamma(U_\lambda \cap U_\mu, \mathcal{L}).
\]
The restriction condition $\mathcal{L}|_{D_1} \cong \mathcal{O}_{D_1}$ implies that on each overlap $U_\lambda \cap U_\mu \cap D_1$, the restricted transition functions can be trivialized by local non-vanishing sections $s_\lambda \in \Gamma(U_\lambda \cap D_1, \mathcal{O}_{D_1}^\times)$ satisfying
\[
g_{\lambda\mu}\big|_{D_1} = \frac{s_\lambda}{s_\mu} \quad \text{on } U_\lambda \cap U_\mu \cap D_1.
\]
Since $U_\lambda$ is affine and $D_1 \subset D_2$ is defined by a nilpotent ideal, the canonical map $\Gamma(U_\lambda, \mathcal{O}_{D_2}^\times) \to \Gamma(U_\lambda \cap D_1, \mathcal{O}_{D_1}^\times)$ is surjective. 
Hence, we can choose local lifts $\tilde{s}_\lambda \in \Gamma(U_\lambda, \mathcal{O}_{D_2}^\times)$ for each $s_\lambda$. 
By replacing $g_{\lambda\mu}$ with the modified transition functions $\frac{\tilde{s}_\mu}{\tilde{s}_\lambda} g_{\lambda\mu}$, we may assume without loss of generality that
\[
g_{\lambda\mu}\big|_{D_1} = 1 \quad \text{for all } \lambda, \mu,
\]
which means $g_{\lambda\mu} - 1 \in \Gamma(U_\lambda \cap U_\mu, \mathcal{I}_{D_1/D_2})$.
The sheaf $\mathcal{I}_{D_1/D_2} \otimes_{\mathcal{O}_{D_2}} \mathcal{L}$ is constructed by gluing the local subsheaves $\mathcal{I}_{D_1/D_2}\big|_{U_\lambda}$ via the multiplication isomorphisms $g_{\lambda\mu} : \mathcal{I}_{D_1/D_2}\big|_{U_\lambda \cap U_\mu} \xrightarrow{\,\sim\,} \mathcal{I}_{D_1/D_2}\big|_{U_\lambda \cap U_\mu}$. 
Since $\mathcal{I}_{D_1/D_2}^2=0$, for any open subset $V \subseteq U_\lambda \cap U_\mu$ and any local section $x \in \Gamma(V, \mathcal{I}_{D_1/D_2})$, we have
\[
g_{\lambda\mu} x = \big(1 + (g_{\lambda\mu} - 1)\big) x = x + (g_{\lambda\mu} - 1) x = x.
\]
Thus, $g_{\lambda\mu}$ acts as the identity automorphism on $\mathcal{I}_{D_1/D_2}\big|_{U_\lambda \cap U_\mu}$ for every pair $\lambda, \mu$, yielding a canonical isomorphism of sheaves $\mathcal{I}_{D_1/D_2} \otimes_{\mathcal{O}_{D_2}} \mathcal{L} \;\cong\; \mathcal{I}_{D_1/D_2}$.
Thus, we obtain the desired exact sequence:
\[
0 \longrightarrow \mathcal{I}_{D_1/D_2} \longrightarrow \mathcal{L} \longrightarrow \mathcal{O}_{D_1} \longrightarrow 0.
\]

\medskip
\noindent
(2) Let $\xi \in \mathrm{Ext}^1_{\mathcal{O}_{D_2}}(\mathcal{O}_{D_1}, \mathcal{I}_{D_1/D_2})$ be the extension class of the exact sequence \eqref{eq: s.e.s. over D2}.
Taking the long exact sequence of the cohomology of \eqref{eq: s.e.s. over D2}, the extension class $\xi$ corresponds to the image of $1 \in  H^0(D_1, \mathcal{O}_{D_1})$ under the connecting homomorphism
\[
\delta : H^0(D_1, \mathcal{O}_{D_1}) \longrightarrow H^1(D_2, \mathcal{I}_{D_1/D_2}),
\]
that is, $\xi = \delta(1) \in H^1(D_2, \mathcal{I}_{D_1/D_2})$.

Explicitly, with respect to the local bases $\tilde{e}_\lambda \in \Gamma(U_\lambda, \mathcal{L})$ chosen in (1) satisfying $\tilde{e}_\lambda\big|_{D_1} = 1 \in \Gamma(U_\lambda \cap D_1, \mathcal{O}_{D_1})$, each $\tilde{e}_\lambda$ serves as a local lift of $1 \in \Gamma(U_\lambda \cap D_1, \mathcal{O}_{D_1})$. 
On each overlap $U_\lambda \cap U_\mu$, the difference between these local lifts is given by
\[
\tilde{e}_\mu - \tilde{e}_\lambda = (g_{\lambda\mu} - 1) \tilde{e}_\lambda \;\in\; \Gamma(U_\lambda \cap U_\mu, \mathcal{I}_{D_1/D_2}).
\]
Thus, the 1-cocycle $\{g_{\lambda\mu} - 1\}_{\lambda\mu}$ represents $\xi = \delta(1) \in H^1(D_2, \mathcal{I}_{D_1/D_2})$ in \v{C}ech cohomology.
On the other hand, since $\mathcal{I}_{D_1/D_2}^2 = 0$, there is a short exact sequence of sheaves:
\[
0 \longrightarrow \mathcal{I}_{D_1/D_2} \xrightarrow{\;\exp\;} \mathcal{O}_{D_2}^\times \xrightarrow{\;\text{res}\;} \mathcal{O}_{D_1}^\times \longrightarrow 1,
\]
where the exponential map $\exp : \mathcal{I}_{D_1/D_2} \to \mathcal{O}_{D_2}^\times$ is given locally by $f \mapsto 1 + f$. 
This short exact sequence induces the connecting map on cohomology groups:
\[
H^1(\exp) : H^1(D_2, \mathcal{I}_{D_1/D_2}) \longrightarrow H^1(D_2, \mathcal{O}_{D_2}^\times) \cong \mathrm{Pic}(D_2).
\]
Applying $H^1(\exp)$ to the extension class $\xi = \delta(1) = [\{g_{\lambda\mu} - 1\}_{\lambda\mu}]$, we obtain:
\[
H^1(\exp)(\xi) = [\{\exp(g_{\lambda\mu} - 1)\}_{\lambda\mu}] = [\{1 + (g_{\lambda\mu} - 1)\}_{\lambda\mu}] = [\{g_{\lambda\mu}\}_{\lambda\mu}] = \mathcal{L} \;\in\; \mathrm{Pic}(D_2),
\]
which proves the claim.
\end{proof}

\begin{proof}[Proof of Theorem~\ref{thm: log CY}]
Let $\rho\colon Y\to X$ be a minimal resolution.
By the arguments in Steps 1--8 of the proof of {\cite[Theorem 3.34]{Tan14}}, it suffices to consider the case where $Y$ is an elliptic ruled surface.

We freely use the notations introduced in Section~\ref{sec: non-vanishing}.
Choose a blow-down sequence 
$$
Y=Y_{N}\to Y_{N-1}\to \cdots \to Y_0
$$
 satisfying Lemma~\ref{claim: replace blow-down seq.}.
By the assumption $\kappa(X, K_{X}+\Delta)=0$ and Lemmas~\ref{lem: non-vanishing 1} and \ref{lem: non-vanishing 2}, we have
$$
\frac{2}{\alpha+2}\Delta_{Y_0} \sim_{\mathbb{Q}} -K_{Y_0} \sim_{\mathbb{Q}} 2 T_{\mathcal{E}} - p^*\mathfrak{e}.
$$
Consequently,
$$
K_{Y_0}+\Delta_{Y_0}\sim_{\Q}-\frac{\alpha}{2}K_{Y_0}\sim_{\Q}\alpha T_{\mathcal{E}}-\frac{\alpha}{2}p^{*}\mathfrak{e}.
$$
Since $(K_{Y_0}+\Delta_{Y_0})^2=0$ and $K_{Y}+\Delta_{Y}$ is nef, we have $a_i=0$ for every $i$.
Hence $K_{Y}+\Delta_{Y}=\varphi^{*}(K_{Y_{0}}+\Delta_{Y_{0}})$.
In particular, $\kappa(K_{Y_0}+\Delta_{Y_0})=\kappa(K_{Y}+\Delta_{Y})=0$.
Therefore, for every sufficiently divisible integer $m>0$,
$$
|m(K_{Y}+\Delta_{Y})|=\{\frac{m\alpha}{\alpha+2}\varphi^{*}\Delta_{Y_0}\}.
$$
Thus, it remains to prove that $\alpha=0$.
 
Suppose to the contrary that $\alpha>0$.
Then 
$$
0=\kappa(K_{Y_0}+\Delta_{Y_0})=\kappa(-K_{Y_0}).
$$
By Lemma~\ref{lem: elliptic ruled}, we may assume that $e=0$ and that either $\mathcal{E}$ is decomposable with $\mathfrak{e}$ non-torsion, or $\mathcal{E}$ is indecomposable and $\operatorname{char} k=0$.

First suppose that $\mathcal{E}$ is decomposable with $\mathfrak{e}$ non-torsion.
As in the proof of Lemma~\ref{lem: non-vanishing 2}, we can write $\Delta_{Y_0}=b_1T_1+b_2T_2$,
where $b_1,b_2\in \Q_{\ge 0}$ with $b_1+b_2=\alpha+2$ and $T_1=T_{\mathcal{E}}$ and $T_2\sim T_{\mathcal{E}}-p^{*}\mathfrak{e}$ are the two disjoint sections of $p$.
On the other hand, since $\Delta_{Y_0}  \sim_{\mathbb{Q}} (\alpha+2) T_{\mathcal{E}} - \frac{\alpha+2}{2}p^*\mathfrak{e}$, we have 
$$
b_1=b_2=\frac{\alpha+2}{2}>1.
$$
Thus, the same argument as in the proof of Lemma~\ref{lem: non-vanishing 2} shows that $\mathfrak{e}\sim_{\Q} 0$,
contradicting the assumption that $\mathfrak{e}$ is non-torsion.

Assume that $\mathcal{E}$ is indecomposable and $\operatorname{char} k=0$.
Then $\mathcal{E}$ is the non-trivial extension
\begin{equation}\label{eq: extension}
0\to \mathcal{O}_B\to \mathcal{E}\to \mathcal{O}_B\to 0.
\end{equation}
Since $\mathcal{O}_{Y_0}(T_{\mathcal{E}})|_{T_{\mathcal{E}}}\cong \mathcal{O}_{T_{\mathcal{E}}}$, 
we have the exact sequence
$$
0\to \mathcal{O}_{Y_0}\to \mathcal{O}_{Y_0}(T_{\mathcal{E}})\to \mathcal{O}_{T_{\mathcal{E}}}\to 0.
$$
Pushing this sequence forward by $p$ gives the extension \eqref{eq: extension}.

We will also need the following description of the same extension on the first infinitesimal neighborhood $2T_{\mathcal{E}}$.
Consider  the commutative diagram
\[
\xymatrix{
 & 0 \ar[d] & 0 \ar[d] & & \\
 & \mathcal{O}_{Y_0}(-T_{\mathcal{E}}) \ar@{=}[r] \ar[d] & \mathcal{O}_{Y_0}(-T_{\mathcal{E}}) \ar[d] & & \\
0 \ar[r] & \mathcal{O}_{Y_0} \ar[r] \ar[d] & \mathcal{O}_{Y_0}(T_{\mathcal{E}}) \ar[r] \ar[d] & \mathcal{O}_{Y_0}(T_{\mathcal{E}})|_{T_{\mathcal{E}}} \ar[r] \ar@{=}[d] & 0 \\
0 \ar[r] & \mathcal{O}_{T_{\mathcal{E}}} \ar[r] \ar[d] & \mathcal{O}_{Y_0}(T_{\mathcal{E}})|_{2T_{\mathcal{E}}} \ar[r] \ar[d] & \mathcal{O}_{Y_0}(T_{\mathcal{E}})|_{T_{\mathcal{E}}} \ar[r] & 0, \\
 & 0 & 0 & & 
}
\]
Pushing forward by $p$ shows that  \eqref{eq: extension} is also obtained from the exact sequence
\begin{equation} \label{eq: s.e.s}
0 \to \mathcal{O}_{T_{\mathcal{E}}} \to \mathcal{O}_{Y_0}(T_{\mathcal{E}})|_{2T_{\mathcal{E}}} \to \mathcal{O}_{Y_0}(T_{\mathcal{E}})|_{T_{\mathcal{E}}}\cong \mathcal{O}_{T_{\mathcal{E}}} \to 0.
\end{equation}
The connecting homomorphism associated with \eqref{eq: s.e.s} is
\[
\delta \colon H^0(T_{\mathcal{E}}, \mathcal{O}_{T_{\mathcal{E}}}) \longrightarrow H^1(T_{\mathcal{E}}, \mathcal{O}_{T_{\mathcal{E}}}) \cong H^1(B, \mathcal{O}_B).
\]
Under the standard identification $\mathrm{Ext}^1(\mathcal{O}_{B}, \mathcal{O}_{B}) \cong H^1(B, \mathcal{O}_{B})$, the class $\delta(1)$ is precisely the extension class $\xi \in \mathrm{Ext}^1(\mathcal{O}_{B}, \mathcal{O}_{B})$ of \eqref{eq: extension}.
Since \eqref{eq: extension} is non-split, we have $\delta(1)\neq 0$.

We now have $\Delta_{Y_0}=(\alpha+2)T_{\mathcal{E}}$.
Let $\widetilde{T}_{\mathcal{E}}$ denote the proper transform of $T_{\mathcal{E}}$ on $Y$.
Note that $\rho$ contracts $\widetilde{T}_{\mathcal{E}}$ together with some $\varphi$-exceptional curves.

Since $K_{X}+\Delta$ is $\Q$-Cartier and $\rho^{*}(K_X+\Delta)=K_{Y}+\Delta_{Y}\sim_{\Q}\alpha \varphi^{*}T_{\mathcal{E}}$,
there exists an open neighborhood $U\subset Y$ of the $\rho$-exceptional locus such that
$\mathcal{O}_{Y}(m \rho^{*}(K_X+\Delta))\cong \varphi^{*}\mathcal{O}_{Y_0}(m\alpha T_{\mathcal{E}})$ is trivial on $U$ for a sufficiently divisible $m>0$.
In particular, for any effective $\rho$-exceptional divisor $E$ on $Y$, 
\begin{equation} \label{eq: trivial on excep}
\left(\varphi^{*}\mathcal{O}_{Y_0}(m\alpha T_{\mathcal{E}})\right)|_{E}\cong \mathcal{O}_{E}.
\end{equation}
We need the following claim (for the definition of chain-connected divisors, see \cite{Kon10} or \cite{Eno23}):
\begin{claim}
There exists a $\rho$-exceptional and chain-connected divisor $T_1$ on $Y$  that agrees with $\varphi^{*}T_{\mathcal{E}}$ in a neighborhood of $\widetilde{T}_{\mathcal{E}}$.
\end{claim}

\begin{proof}
We construct effective divisors $T_{1}^{(i)}$ on $Y_i$ inductively.
Set $T_{1}^{(0)}:=T_{\mathcal{E}}$.
Suppose that $i>0$ and $T_{1}^{(i-1)}$ has been constructed.
If the center of the blow-up $\varphi_{i}\colon Y_i\to Y_{i-1}$ is either contained in the proper transform of $T_{\mathcal{E}}$ or does not belong to the support of $T_{1}^{(i-1)}$,
set 
$$
T_{1}^{(i)}:=\varphi_{i}^{*}T_{1}^{(i-1)}.
$$
Otherwise, set 
$$
T_{1}^{(i)}:=\varphi_{i}^{*}T_{1}^{(i-1)}-E_{i}.
$$
Finally, put $T_{1}:=T_{1}^{(N)}$.

By {\cite[Proposition~3.7]{Eno23}} and {\cite[Lemma~3.2~(6)]{Kon10}},
it follows that each $T^{(i)}_{1}$ is chain-connected.
Since $\alpha$ is positive, $\Delta_{Y_i}=\varphi_{i}^{*}\Delta_{Y_{i-1}}-E_i$ and $\Delta_{Y_0}=(\alpha+2)T_{\mathcal{E}}$, it follows by induction on $i$ that,
for every irreducible component $C$ of $T_{1}^{(i)}$,
$$
\operatorname{mult}_{C}\Delta_{Y_i}>\operatorname{mult}_{C}T_{1}^{(i)}.
$$
Hence every irreducible component $C$ of $T_{1}$ satisfies $\operatorname{mult}_{C}\Delta_{Y}>1$.
Since the coefficients of $\Delta$ are at most $1$,
This implies that $T_{1}$ is $\rho$-exceptional.
Finally, by the definition of $T_1^{(i)}$, the subtraction of $E_i$
occurs only when the center of $\varphi_i$ lies on
$\operatorname{Supp}T_1^{(i-1)}$ away from the proper transform of
$T_{\mathcal{E}}$. Thus, near $\widetilde{T}_{\mathcal{E}}$, no such
subtraction occurs. Consequently,
\[
T_1=\varphi^*T_{\mathcal{E}}
\]
in a neighborhood of $\widetilde{T}_{\mathcal{E}}$.
This proves the claim.
\end{proof}

We now take $T_{1}$ as in the above claim, and set 
$$
T_{2}:=T_{1}+\widetilde{T}_{\mathcal{E}}.
$$
Both $T_1$ and $T_2$ are $\rho$-exceptional.
Moreover, the inclusion $T_{1}\subset T_{2}$ is a first order thickening, whose defining ideal sheaf is
$$
\mathcal{I}_{T_1/T_2} \cong \mathcal{O}_{Y}(-T_1)|_{\widetilde{T}_{\mathcal{E}}} 
\cong \mathcal{O}_{Y}(-\varphi^{*}T_{\mathcal{E}})|_{\widetilde{T}_{\mathcal{E}}} 
\cong \varphi^{*} \mathcal{O}_{Y_0}(-T_{\mathcal{E}})|_{T_{\mathcal{E}}} 
\cong \varphi^{*} \mathcal{I}_{T_{\mathcal{E}}/2T_{\mathcal{E}}} 
\cong \mathcal{O}_{\widetilde{T}_{\mathcal{E}}}.
$$
The morphism $\varphi$ induces a morphism of thickenings $(T_1\subset T_2)\to (T_{\mathcal{E}}\subset 2T_{\mathcal{E}})$, and hence
a commutative diagram 
\[
\xymatrix@C=40pt{
 0 \ar[r] & H^{1}(\mathcal{O}_{T_{\mathcal{E}}}) \ar^{\scriptstyle H^1(\exp)}[r] \ar^{\cong}[d] & \operatorname{Pic}(2T_{\mathcal{E}}) \ar[r] \ar^{\varphi^{*}}[d] & \operatorname{Pic}(T_{\mathcal{E}}) \ar[r] \ar^{\varphi^{*}}[d] & 0 \\
0 \ar[r] & H^{1}(\mathcal{O}_{\widetilde{T}_{\mathcal{E}}}) \ar^{\scriptstyle H^1(\exp)}[r]  &  \operatorname{Pic}(T_{2}) \ar[r] & \operatorname{Pic}(T_{1}) \ar[r] & 0. 
}
\]
Here the exactness on the left follows from the the fact that
$$
H^{0}(\mathcal{O}_{T_{\mathcal{E}}}^{\times})\cong k^{\times},\quad H^{0}(\mathcal{O}_{T_{1}}^{\times})\cong k^{\times},
$$
which in turn follows from the chain-connectivity of $T_{\mathcal{E}}$ and $T_{1}$ (see {\cite[Lemma~3.9]{Eno23}}).
Thus,
$$
H^{0}(\mathcal{O}_{2T_{\mathcal{E}}}^{\times}) \twoheadrightarrow H^{0}(\mathcal{O}_{T_{\mathcal{E}}}^{\times}),\quad H^{0}(\mathcal{O}_{T_{2}}^{\times}) \twoheadrightarrow H^{0}(\mathcal{O}_{T_{1}}^{\times}).
$$

By the construction of the extension class $\xi=\delta(1)$ and Lemma~\ref{lem: 1st order thickening},
 the line bundle $\mathcal{O}_{Y_0}(m\alpha T_{\mathcal{E}})|_{2T_{\mathcal{E}}}$ is the image of 
  $m\alpha \xi$ under  $H^{1}(\exp)$.
 By the commutativity of the above diagram,
 its pullback $\left(\varphi^{*}\mathcal{O}_{Y_0}(m\alpha T_{\mathcal{E}})\right)|_{T_{2}}$ is also the image of $m\alpha \xi$ under the bottom $H^{1}(\exp)$.
 Since $m\alpha>0$ and $\xi\neq 0$, we have $m\alpha \xi\neq 0$.
 Moreover,  since $H^{1}(\exp)$ is injective, $\varphi^{*}\mathcal{O}(m\alpha T_{\mathcal{E}})|_{2T_{\mathcal{E}}}$ is non-trivial.
 On the other hand, $T_2$ is $\rho$-exceptional, so \eqref{eq: trivial on excep} gives 
 $\left(\varphi^{*}\mathcal{O}_{Y_0}(m\alpha T_{\mathcal{E}})\right)|_{T_2}\cong \mathcal{O}_{T_2}$,
 a contradiction.
Therefore, $\alpha=0$, and the proof is complete.
\end{proof}

\section{Non-big case}\label{sec: non-big case}
In this section, we complete the proof of the abundance theorem when  $K_X+\Delta$ is not $\pi$-big.
When $\dim S=0$, we may assume that $\kappa(X, K_X+\Delta)=1$ by Theorems~\ref{thm: non-vanishing} and \ref{thm: log CY}.
The abundance in this case was already proved in {\cite[Theorem 4.1]{Fuj84}}.
Thus, we may assume that $S$ is a curve.
Here, we prove the following slightly stronger statement:

\begin{thm}\label{lem: abandance non-big}
Let $\pi\colon X\to S$ be a proper surjective morphism from a normal surface $X$ to a curve $S$ over $k$.
Let $\Delta$ be an effective $\Q$-divisor on $X$ such that $K_X+\Delta$ is $\pi$-nef but not $\pi$-big.
Then, $K_X+\Delta$ is $\Q$-Cartier and $\pi$-semiample.
\end{thm}
For this case, the proof parallels \cite[Corollary 4.4]{Mor26}.
To adapt the argument to arbitrary characteristic, we slightly modify the proof to avoid relying on Sard's theorem for the smoothness of general fibers.
For the reader's convenience, we give the complete details below.
\begin{proof}[Proof of Theorem \ref{lem: abandance non-big}]
Since semi-ampleness is local on $S$, we may replace $S$ with an affine open neighborhood.
Taking the Stein factorization of $\pi \colon X \to S$, we may assume that $S$ is a smooth affine curve and $\pi$ has connected fibers.
To prove the theorem, 
it suffices to show that $K_{X}+\Delta\sim_{\Q,\pi}0$.
Replacing $X$ with its minimal resolution, we may assume that $X$ is smooth.
%Note that $\Delta$ remains effective, although it may no longer be a boundary divisor.

Since $X$ is smooth and $S$ is a smooth curve, the morphism $\pi$ is flat and every closed fiber $X_t = \pi^{-1}(t)$ is an effective Cartier divisor on $X$.
Moreover, since $\pi_* \mathcal{O}_X = \mathcal{O}_S$, the generic fiber of $\pi$ is geometrically integral. 
Thus, any general closed fiber $X_t$ is an integral Gorenstein curve with dualizing sheaf $\omega_{X_t} \cong \mathcal{O}_X(K_X)|_{X_t}$.
We may also assume that $X_t$ shares no common components with $\Delta$, so $\Delta|_{X_t}$ defines an effective $\mathbb{Q}$-divisor on $X_t$.

Since $\dim S = 1$ and $K_X + \Delta$ is not $\pi$-big, we have
\[
\deg \left((K_X + \Delta)|_{X_t}\right) = 0.
\]
For a sufficiently divisible positive integer $m$, let $\mathcal{L}_m := \mathcal{O}_{X_t}(m(K_X + \Delta)|_{X_t})$.
By the Riemann--Roch theorem for Gorenstein curves, we have
\[
\chi(X_t, \mathcal{L}_m) = \deg(\mathcal{L}_m) + 1 - p_a(X_t) = 1 - p_a(X_t),
\]
where $p_a(X_t) = h^1(X_t, \mathcal{O}_{X_t})$ is the arithmetic genus of $X_t$.
We consider two cases:

\medskip
\noindent
\textit{Case 1: $p_a(X_t) = 0$.}
    In this case, $\chi(X_t, \mathcal{L}_m) = 1$, and hence $h^0(X_t, \mathcal{L}_m) \ge 1$.
    Indeed, $X_t\cong \mathbb{P}^1$ and $\mathcal{L}_m\cong \mathcal{O}_{X_t}$.

\medskip
\noindent
\textit{Case 2: $p_a(X_t) \ge 1$.}
By Serre duality on the Gorenstein curve $X_t$, we have $h^0(X_t, \omega_{X_t}) = h^1(X_t, \mathcal{O}_{X_t}) = p_a(X_t) \ge 1$. Since $\deg(\omega_{X_t}) = 2 p_a(X_t) - 2$ and $\Delta|_{X_t} \ge 0$, we observe that
\[
0 = \deg\left((K_X + \Delta)|_{X_t}\right) = \deg(\omega_{X_t}) + \deg(\Delta|_{X_t}) = (2 p_a(X_t) - 2) + \deg(\Delta|_{X_t}).
\]
Since $p_a(X_t) \ge 1$ and $\deg(\Delta|_{X_t}) \ge 0$, we must have $p_a(X_t) = 1$, $\deg(\omega_{X_t}) = 0$, and $\Delta|_{X_t} = 0$. Since $\omega_{X_t}$ is a line bundle of degree $0$ with a non-zero section, we obtain $\omega_{X_t} \cong \mathcal{O}_{X_t}$. Thus, $\mathcal{L}_{m} \cong \mathcal{O}_{X_t}$, which yields $h^0(X_t, \mathcal{L}_m) = 1 > 0$.

\medskip
In both cases, we obtain $h^0\left(X_t, \mathcal{O}_{X_t}(m(K_X + \Delta)|_{X_t})\right)>0$.
By the semi-continuity theorem, $\pi_* \mathcal{O}_X(m(K_X + \Delta)) \neq 0$.
Since $S$ is affine, there exists an effective Cartier divisor $D \ge 0$ on $X$ such that $m(K_X + \Delta) \sim D$.

Since $\deg (D|_{X_t})=\deg(m(K_X + \Delta)|_{X_t})=0$, the divisor $D$ consists of $\pi$-vertical components.
In particular, by {\cite[Corollary~2.6]{Bad01})}, $D^2 \le 0$.
On the other hand, since $K_X+\Delta$ is $\pi$-nef, we have
$$
D^2=m(K_X+\Delta)\cdot D\ge 0.
$$
Therefore, $D^2 = 0$.
By the equality case of Zariski's lemma, there exists an effective $\Q$-divisor $E$ on $S$ such that $D = \pi^* E$.
This completes the proof.
\end{proof}

\section{Vanishing}\label{sec: vanishing}

In this section, we recall the vanishing theorem of \cite{Eno24}, which will be used in the proof of the abundance theorem 
in the big case.
Let $\pi\colon X\to S$ be a proper surjective morphism from a normal surface $X$ to a variety $S$ over $k$.
Let $\Delta$ be a boundary $\mathbb{R}$-divisor on $X$.

\begin{defn}
An $\mathbb{R}$-divisor $D$ on $X$ is said to be \emph{$\pi$-$\mathbb{Z}$-positive} if for every $\pi$-exceptional negative definite $\mathbb{Z}$-divisor $B = \sum a_i C_i > 0$ (i.e., $\dim \pi(\operatorname{Supp} B) = 0$ and the intersection matrix $(C_i \cdot C_j)$ is negative definite in the sense of Mumford's intersection form), there exists an irreducible component $C_i$ of $B$ such that 
\[
(D - B) \cdot C_i > 0.
\]
\end{defn}

\begin{thm}[{$\mathbb{Z}$-Zariski decomposition, \cite[Theorem 3.5]{Eno24}}] \label{thm: Z-Zariski}
For any $\pi$-pseudo-effective $\mathbb{R}$-divisor $D$ on $X$, there exists a unique decomposition 
\[
D = P_{\mathbb{Z}} + N_{\mathbb{Z}},
\]
called the $\mathbb{Z}$-Zariski decomposition of $D$, such that:
\begin{enumerate}
    \item[\textup{(i)}] $P_{\mathbb{Z}}$ is $\pi$-$\mathbb{Z}$-positive,
    \item[\textup{(ii)}] $N_{\mathbb{Z}} = 0$, or $N_{\mathbb{Z}} > 0$ is a $\pi$-exceptional negative definite $\mathbb{Z}$-divisor, and
    \item[\textup{(iii)}] $P_{\mathbb{Z}} \cdot C \le 0$ for every irreducible component $C$ of $N_{\mathbb{Z}}$.
\end{enumerate}
\end{thm}

\begin{thm}[{Vanishing Theorem, \cite[Theorem 4.1]{Eno24}}] \label{thm: vanishing}
Let $D$ be a $\pi$-big $\mathbb{Z}$-divisor, and let $D = P_{\mathbb{Z}} + N_{\mathbb{Z}}$ be its $\mathbb{Z}$-Zariski decomposition. Then
\[
R^1\pi_* \mathcal{O}_X(K_X + D) \cong H^1(N_{\mathbb{Z}}, \mathcal{L}_D),
\]
where $\mathcal{L}_D := \operatorname{Coker}\left( \mathcal{O}_X(K_X + P_{\mathbb{Z}}) \hookrightarrow \mathcal{O}_X(K_X + D) \right)$, unless $\dim Y = 0$ and $\operatorname{char} k > 0$.
\end{thm}

\begin{rem} \label{rem: vanishing in positive char}
When $\dim S = 0$ and $\operatorname{char} k > 0$, the same assertion holds under the assumption that 
\[
\dim |D| \ge \dim_k \left( \ker (F^m \colon H^1(\mathcal{O}_X) \to H^1(\mathcal{O}_X)) \right) \quad \text{for } m \gg 0,
\]
by {\cite[Theorem 4.17]{Eno23}}.
\end{rem}

In the remainder of this section, we reformulate this vanishing theorem in a form that is convenient for the proof of the abundance theorem 
in the big case.

\begin{defn}
Let $M$ be a $\pi$-nef $\mathbb{R}$-divisor on $X$ (not necessarily $\mathbb{R}$-Cartier). An irreducible component $C$ of $\Delta$ is called an \emph{$M$-trivial lc place} if
\[
\operatorname{coeff}_C \Delta = 1, \quad C \text{ is } \pi\text{-exceptional}, \quad \text{and} \quad M \cdot C = 0.
\]
\end{defn}

\begin{lem}
Let $M$ be a $\pi$-nef $\mathbb{R}$-divisor and put $D:=\Delta+M$. 
Then the $\mathbb{Z}$-Zariski decomposition of $D$ can be written as $D = P_{\mathbb{Z}} + N_{\mathbb{Z}}$, where
\[
N_{\mathbb{Z}} := \sum_{\lambda} \Delta_\lambda
\]
is the sum of all connected components $\Delta_\lambda$ of $\Delta$ such that:
\begin{itemize}
    \item every irreducible component of $\Delta_\lambda$ is an $M$-trivial lc place, and
    \item $\Delta_\lambda$ is negative definite (which holds automatically if $M$ is $\pi$-big).
\end{itemize}
\end{lem}

\begin{proof}
We check conditions (i)--(iii) in Theorem~\ref{thm: Z-Zariski} for the $\mathbb{Z}$-Zariski decomposition.
Condition (ii) is clear from the construction of $N_{\mathbb{Z}}$.
For (iii), take any irreducible component $C \le N_{\mathbb{Z}}$. We have
\[
-P_{\mathbb{Z}} \cdot C = (N_{\mathbb{Z}} - D) \cdot C = (N_{\mathbb{Z}} - \Delta) \cdot C - M \cdot C = 0,
\]
since $C$ and $\Delta-N_{\mathbb{Z}}$ do not intersect and $C$ is an $M$-trivial lc place. Thus $P_{\mathbb{Z}} \cdot C = 0$.
For (i), replacing $\Delta$ by $\Delta - N_{\mathbb{Z}}$, it suffices to show that $D$ is $\pi$-$\mathbb{Z}$-positive under the assumption $N_{\mathbb{Z}} = 0$. That is, for any negative definite $\pi$-exceptional $\mathbb{Z}$-divisor $B > 0$, there exists an effective divisor $C$ such that $\operatorname{Supp} C \subseteq \operatorname{Supp} B$ and $(D - B) \cdot C > 0$.
Let $B = \sum_i n_i C_i$ be the irreducible decomposition. Define
\[
B' := \sum_{n_i \ge 2 \text{ or } \operatorname{coeff}_{C_i}\Delta < 1} n_i C_i \quad \text{and} \quad B'' := B - B'.
\]
Let $\Delta'$ be the maximal subdivisor of $\Delta$ such that $\operatorname{Supp} \Delta' \subseteq \operatorname{Supp} B'$, and set $\Delta'' := \Delta - \Delta'$. Then we can write
\[
D - B = (\Delta' - B') + (\Delta'' - B'') + M.
\]
By construction, the following properties hold:
\begin{itemize}
    \item $\Delta' - B' \le 0$, with equality holding if and only if $B' = 0$.
    \item  $\Delta'' - B'' \ge 0$, and it shares no components with $B$.
\end{itemize}
Now we divide into two cases.

If $B' > 0$, set $C := B' - \Delta' > 0$. Then we have
    \[
    (D - B) \cdot C = -C^2 + (\Delta'' - B'') \cdot C + M \cdot C > 0,
    \]
    since $-C^2 > 0$, $(\Delta'' - B'') \cdot C \ge 0$, and $M \cdot C \ge 0$.

If $B' = 0$, then $B = B'' > 0$. Suppose to the contrary that $(D - B) \cdot B \le 0$. Since
    \[
    (D - B) \cdot B = (\Delta'' - B'') \cdot B'' + M \cdot B \ge 0
    \]
    (owing to $(\Delta'' - B'') \cdot B'' \ge 0$ and $M \cdot B \ge 0$), we must have $(D - B) \cdot B = 0$. This implies 
    \[
    \operatorname{Supp} B \cap \operatorname{Supp}(\Delta - B) = \emptyset
    \]
    and $M$ is numerically trivial on $B$. Thus, each connected component of $B$ is a connected component of $\Delta$, and every irreducible component of $B$ is an $M$-trivial lc place. Consequently, $N_{\mathbb{Z}} > 0$, which contradicts the hypothesis $N_{\mathbb{Z}} = 0$.
    Therefore, taking $C:=B$, we obtain $(D - B) \cdot C > 0$, completing the proof of (i).
\end{proof}

Combining this with Theorem~\ref{thm: vanishing} and Remark~\ref{rem: vanishing in positive char}, we obtain the following corollary:
\begin{cor}\label{cor: vanishing}
Let $M$ be a $\pi$-nef and $\pi$-big $\mathbb{R}$-divisor such that $D := \Delta + M$ is a $\mathbb{Z}$-divisor. Assume that $\Delta$ has no connected components consisting entirely of $M$-trivial lc places. 
If $\dim S = 0$ and $\operatorname{char} k > 0$, we additionally assume 
$$
\dim |D| \ge \dim_k \left( \ker (F^m \colon H^1(\mathcal{O}_X) \to H^1(\mathcal{O}_X)) \right)
$$
for $m \gg 0$.
Then
\[
R^1 \pi_* \mathcal{O}_X(K_X + D) = 0.
\]
\end{cor}

\section{Big case} \label{sec: big case}
In this section, we prove the abundance theorem for the big case:

\begin{thm}\label{thm: abundance big case}
Let $\pi\colon X\to S$ be a proper surjective morphism from a normal surface $X$ to a variety $S$ over $k$.
Let $\Delta$ be a boundary $\mathbb{Q}$-divisor on $X$ such that $K_X+\Delta$ is $\mathbb{Q}$-Cartier.
If $K_X+\Delta$ is $\pi$-nef and $\pi$-big, then it is $\pi$-semiample.
\end{thm}

To prove this, we use the following adjunction result, which generalizes {\cite[Lemma~4.4]{Fuj12}} from integral curves to reduced curves:

\begin{lem}[{Adjunction}] \label{lem: adjunction}
Let $X$ be a normal surface, and $\Delta_\lambda$ a reduced closed subscheme of pure dimension one. 
Then we have the short exact sequence
$$
0 \to \mathcal{T} \to \omega_X(\Delta_\lambda)\otimes\mathcal{O}_{\Delta_\lambda}\to \omega_{\Delta_\lambda} \to 0
$$
where $\mathcal{T}$ is a sheaf whose support has dimension zero.
\end{lem}

\begin{proof}
%For the detailed proof, see Section \ref{sec: duality}.
Since $X$ is a normal surface, it satisfies Serre's condition $(S_2)$ and is thus Cohen--Macaulay. 
Similarly, $\Delta_\lambda$ is Cohen--Macaulay because it is a reduced curve ($R_0 + S_1$). 
Using the theory of dualizing complexes, we obtain
\[
\omega_{\Delta_\lambda} \cong \mathcal{E}xt^1_{\mathcal{O}_X}(\mathcal{O}_{\Delta_\lambda}, \omega_X).
\]
With this duality at hand, the assertion follows by applying exactly the same argument as in {\cite[Lemma~4.4]{Fuj12}}. 
%In the complex analytic setting, the assertion can be proved by applying the duality theory for complex analytic spaces developed in \cite{RRV71}.
\end{proof}

Take $m\in\Z_{>0}$ such that $m(K_X+\Delta)$ is Cartier.
Put 
$$
M:=(m-1)(K_X+\Delta),\quad D:=m(K_X+\Delta)-K_X=M+\Delta,
$$
where $M$ is a $\pi$-nef and $\pi$-big $\Q$-Cartier $\Q$-divisor and $D$ is a $\mathbb{Z}$-divisor.
Let $\Delta_M$ be the sum of all $M$-trivial lc places of $\Delta$.
Let $\Delta_M=\sum_\lambda\Delta_\lambda$ be the connected component decomposition.

\begin{lem}\label{lem: restriction is trivial} $\mathcal{O}_X(m(K_X+\Delta))|_{\Delta_\lambda}\cong\mathcal{O}_{\Delta_\lambda}$
\end{lem}
\begin{proof}%[Proof of Lemma \ref{lem: restriction is trivial}]

\setcounter{stepcounter}{0}
\step \label{step: non-zero is enough}
We first show that $H^0(\Delta_\lambda, \mathcal{O}_X(K_X+D)|_{\Delta_\lambda})\neq 0$ implies $\mathcal{O}_X(K_X+D)|_{\Delta_\lambda}\cong\mathcal{O}_{\Delta_\lambda}$.
Let 
$$
\nu\colon \widetilde{\Delta}_{\lambda}=\bigsqcup_{i} \widetilde{C}_{i}
\to\Delta_\lambda=\bigcup_{i} C_i
$$
be the normalization. 
Thus,
\begin{align*}
H^0(\Delta_\lambda, \mathcal{O}_X(K_X+D)|_{\Delta_\lambda})
&\hookrightarrow
H^0(\widetilde{\Delta}_{\lambda}, \mathcal{O}_X(K_X+D)|_{\widetilde{\Delta}_{\lambda}})\\
&\cong \oplus_i
H^0(\widetilde{C}_{i}, \mathcal{O}_X(K_X+D)|_{\widetilde{C}_{i}}).
\end{align*}
Since each $C_i$ is an $M$-trivial lc place, we have
 $$
 \deg \nu^*(K_X+D)|_{\widetilde{C}_{i}}=(K_X+D)\cdot C_i=0.
 $$
This implies that either $\mathcal{O}_X(K_X+D)|_{\widetilde{C}_{i}}\cong\mathcal{O}_{\widetilde{C}_{i}}$ or $H^0(\widetilde{C}_{i}, \mathcal{O}_X(K_X+D)|_{\widetilde{C}_{i}})=0$.
If $H^0(X, \mathcal{O}_X(K_X+D)|_{\Delta_\lambda})$ has a non-zero global section $s$, then $s|_{\widetilde{C}_{i}}\neq 0$ for some $i$.
Thus, $s|_{\widetilde{C}_{i}}$ is a non-zero constant.
If $C_j$ intersects $C_i$, then the two restrictions
$s|_{\widetilde{C}_i}$ and $s|_{\widetilde{C}_j}$ agree at a point lying
over $C_i\cap C_j$. 
Since $s|_{\widetilde{C}_i}\neq0$, it follows that $s|_{\widetilde{C}_j}\neq 0$.
Since $\Delta_\lambda$ is connected, the section $\nu^*s$ nowhere vanishes. Since $\nu$ is surjective, so does $s$.
Therefore, we have $\mathcal{O}_X(K_X+D)|_{\Delta_\lambda}\cong\mathcal{O}_{\Delta_\lambda}$.

\step
We show that $H^0(\Delta_\lambda, \omega_{\Delta_\lambda})\neq 0$ implies $\mathcal{O}_X(K_X+D)|_{\Delta_\lambda}\cong\mathcal{O}_{\Delta_\lambda}$.
By Lemma \ref{lem: adjunction}, the natural map
$$
H^0(\Delta_\lambda, \omega_X(\Delta_\lambda)|_{\Delta_\lambda})
\twoheadrightarrow
H^0(\Delta_\lambda, \omega_{\Delta_\lambda})
$$
is surjective. 
Hence, since $H^0(\Delta_\lambda, \omega_{\Delta_\lambda})\neq 0$, we can choose a non-zero section $s\in H^0(\Delta_\lambda, \omega_X(\Delta_\lambda)|_{\Delta_\lambda})$. 
Let $t\in H^0(X, \mathcal{O}_X(m(\Delta-\Delta_\lambda)))$ be the canonical section corresponding to the effective divisor $m(\Delta-\Delta_\lambda)$.
Restricting $t$ to $\Delta_\lambda$ and taking its tensor product with $s^{\otimes m}$, we obtain a section 
\begin{align*}
(t|_{\Delta_\lambda}\otimes s^{\otimes m})^{**} &\in 
H^0\left(\Delta_\lambda, \left(\mathcal{O}_X(m(\Delta-\Delta_\lambda))\otimes\omega_X(\Delta_\lambda)^{\otimes m}\right)^{**}|_{\Delta_\lambda}\right) \\
&= H^0(\Delta_\lambda, \mathcal{O}_X(m(K_X+\Delta))|_{\Delta_\lambda}).
\end{align*}
This section is non-zero, since
 $t$ is non-zero on each irreducible component of $\Delta_{\lambda}$ and $s^{\otimes m}$ is non-zero.
Therefore, by Step \ref{step: non-zero is enough}, the proof is complete.

\step
Finally, we show that $H^0(\Delta_\lambda, \omega_{\Delta_\lambda}) = 0$ also implies $\mathcal{O}_X(K_X+D)|_{\Delta_\lambda}\cong\mathcal{O}_{\Delta_\lambda}$.
By Serre duality, we have $H^1(\Delta_\lambda, \mathcal{O}_{\Delta_\lambda}) = 0$.
By \cite[Section 8.4, Theorem 1]{BLR90}, the Picard functor of $\Delta_\lambda$ is representable by a group scheme, and its Lie algebra is isomorphic to $H^1(\Delta_\lambda, \mathcal{O}_{\Delta_\lambda})$.
Since $H^1(\Delta_\lambda, \mathcal{O}_{\Delta_\lambda}) = 0$, we have $\operatorname{Pic}^0(\Delta_\lambda) = 0$.
Since $\mathcal{O}_X(K_X+D)|_{\Delta_\lambda}$ lies in $\operatorname{Pic}^0(\Delta_\lambda) = 0$, we conclude that 
$$
\mathcal{O}_X(K_X+D)|_{\Delta_\lambda} \cong \mathcal{O}_{\Delta_\lambda}.
$$
\end{proof}

%\begin{defn}
We recall that the base locus of the relative linear system $|L/S|$ for a Cartier divisor $L$ on $X$ is defined as 
$$
\operatorname{Bs}|L/S|:=
\operatorname{Supp}\operatorname{Coker}(\pi^*\pi_*\mathcal{O}_X(L)\rightarrow\mathcal{O}_X(L)))\subset X.
$$
%\end{defn}

\begin{lem}\label{lem: component of coeff 1}
Let $C$ be an irreducible component of $\lfloor \Delta \rfloor$.
%When $\operatorname{dim}Y=0$ and $\operatorname{char}(k)>0$, we further assume that $m \gg 0$.
Then for $m\gg 0$, the following hold:
\begin{itemize}
    \item[$(1)$] If $C$ is an $M$-trivial lc place, then $\operatorname{Bs}|m(K_X+\Delta)/S|\cap C=\emptyset$.
    \item[$(2)$] $C\nsubseteq\operatorname{Bs}|m(K_X+\Delta)/S|$.
\end{itemize}
\end{lem}
\begin{proof}%[Proof of Lemma \ref{lem: component of coeff 1}]
We have
$$
D-\Delta_M=\Delta-\Delta_M+(m-1)(K_X+\Delta),
$$
where $\Delta-\Delta_M$ has no $M$-trivial lc places.
By Corollary \ref{cor: vanishing}, it follows that
$$
R^1\pi_*\mathcal{O}_X(K_X+D-\Delta_M)=0.
$$
Considering the short exact sequence
$$
0\to\mathcal{O}_X(K_X+D-\Delta_M)\to\mathcal{O}_X(K_X+D)\to\mathcal{O}_X(K_X+D)|_{\Delta_M}\to 0,
$$
we obtain a surjection
$$
\pi_*\mathcal{O}_X(K_X+D)\twoheadrightarrow \pi_*(\mathcal{O}_X(K_X+D)|_{\Delta_M}).
$$
We consider the commutative diagram
$$
\xymatrix{
\pi^*\pi_*\mathcal{O}_X(K_X+D) \ar[r] \ar[d]&
\mathcal{O}_X(K_X+D) \ar[d]
\\
\pi^*\pi_*(\mathcal{O}_X(K_X+D)|_{\Delta_M})\ar[r]&
\mathcal{O}_X(K_X+D)|_{\Delta_M}
}
$$
We note that $\pi^*\pi_*\mathcal{O}_X(K_X+D) \rightarrow \pi^*\pi_*(\mathcal{O}_X(K_X+D)|_{\Delta_M})$ is surjective because $\pi^*$ is a right exact functor.
By Lemma \ref{lem: restriction is trivial},
$\mathcal{O}_X(K_X+D)|_{\Delta_M}\cong\mathcal{O}_{\Delta_M}$.
In particular,
$\pi^*\pi_*(\mathcal{O}_X(K_X+D)|_{\Delta_M})\to 
\mathcal{O}_X(K_X+D)|_{\Delta_M}$ is surjective.
Then $\pi^*\pi_*\mathcal{O}_X(K_X+D) \to
\mathcal{O}_X(K_X+D)$ is surjective along $\Delta_{M}$.
Therefore, $\operatorname{Bs}|m(K_X+\Delta)/S|\cap \Delta_M=\emptyset$.
This proves the first statement (1).

We next prove the second statement (2).
Take an irreducible component $C$ of $\lfloor \Delta \rfloor$ that is not an $M$-trivial lc place.
If $C$ intersects with $\Delta_{M}$, then the assertion follows from (1).
Thus, we may assume that $C\cap\Delta_M=\emptyset$.
Since $\Delta-\Delta_M-C$ has no $M$-trivial lc places, by Corollary~\ref{cor: vanishing}, it follows that
$$
R^1\pi_*\mathcal{O}_X(K_X+D-\Delta_M-C)=0.
$$
Note that $K_X+D-\Delta_M$ is Cartier around $C$ since $C$ is disjoint from $\Delta_M$.
Considering the short exact sequence
$$
0\to\mathcal{O}_X(K_X+D-\Delta_M-C)\to\mathcal{O}_X(K_X+D-\Delta_M)\to\mathcal{O}_X(K_X+D-\Delta_M)|_C\to 0,
$$
we obtain a surjection
$$
\pi_*\mathcal{O}_X(K_X+D-\Delta_M)\twoheadrightarrow \pi_*(\mathcal{O}_X(K_X+D-\Delta_M)|_C).
$$
We consider the commutative diagram
$$
\xymatrix{
\pi_*\mathcal{O}_X(K_X+D-\Delta_M) \ar[r] \ar[d]&
\pi_*(\mathcal{O}_X(K_X+D-\Delta_M)|_C) \ar[d]^{\cong}
\\
\pi_*\mathcal{O}_X(K_X+D)\ar[r]&
\pi_*(\mathcal{O}_X(K_X+D)|_C)
}
$$
which yields the surjection
$\pi_*\mathcal{O}_X(K_X+D)\to
\pi_*(\mathcal{O}_X(K_X+D)|_C)$.
When $\pi(C)$ is a point, then $(K_X+\Delta)\cdot C>0$ since $C$ is not an $M$-trivial lc place.
Thus, $\mathcal{O}_X(K_X+D)|_C$ is ample and we may assume that $\pi^*\pi_*(\mathcal{O}_{X}(m(K_X+\Delta))|_C)\to\mathcal{O}_{X}(m(K_X+\Delta))|_C$ is surjective after replacing $m$ if necessary.
When $\pi(C)$ is a curve, $\pi|_{C}$ is a finite morphism. Then $\pi^*\pi_*(\mathcal{O}_{X}(m(K_X+\Delta))|_C)\to\mathcal{O}_{X}(m(K_X+\Delta))|_C$  is surjective.
In both cases, 
$$
\pi^*\pi_*\mathcal{O}_X(K_X+D)\to
\pi^*\pi_*(\mathcal{O}_X(K_X+D)|_C)\to\mathcal{O}_{X}(m(K_X+\Delta))|_C
$$
is surjective.
In particular,
$C\nsubseteq\operatorname{Bs}|m(K_X+\Delta)/S|$
(or more strongly, $C\cap \operatorname{Bs}|m(K_X+\Delta)/S|=\emptyset$ if $C$ is disjoint from $\Delta_{M}$).
\end{proof}

\begin{lem}\label{lem: 0-dim non-klt}
Let $x\in X$ be a zero-dimensional non-klt center of $(X, \Delta)$.
%When $\operatorname{dim}Y=0$ and $\operatorname{char}(k)>0$, we further assume that $m \gg 0$.
Then  $x\notin\operatorname{Bs}|m(K_X+\Delta)/S|$ for $m\gg 0$.
\end{lem}
\begin{proof}%[Proof of Lemma \ref{lem: 0-dim non-klt}]
Take a resolution $\rho\colon Y\to X$ at $x$.
Write $K_Y+\Delta_Y=\rho^*(K_X+\Delta)$.
By Lemma \ref{lem: component of coeff 1}, we may assume that $x\notin \operatorname{Supp}\Delta_M$.
Put
$$
Z:=\lfloor\Delta_Y\rfloor-\rho_*^{-1}(\lfloor \Delta \rfloor-\Delta_M).
$$
The boundary $\Q$-divisor
$$
\Delta_Y-Z=\{\Delta_Y\}+\rho_*^{-1}(\lfloor \Delta \rfloor-\Delta_M)
$$
has no $\rho^*M$-trivial lc places.
Using Corollary~\ref{cor: vanishing} on
$$
\rho^*(m(K_X+\Delta))-Z-K_{Y}=\Delta_Y-Z+\rho^*M,
$$
we have
\begin{align}
R^1(\pi\circ\rho)_*\mathcal{O}_{Y}(\rho^*(m(K_X+\Delta))-Z)=0.
\end{align}
By the Leray spectral sequence,
\begin{align*}
R^1\pi_*\mathcal{O}_X(m(K_X+\Delta))\otimes\rho_*\mathcal{O}_Y(-Z))=R^1\pi_*(\rho_*(\mathcal{O}_Y(\rho^*(m(K_X+\Delta))-Z))=0.
\end{align*}
Hence, 
$$
\pi^*\pi_*\mathcal{O}_X(m(K_X+\Delta)) \to
\pi^*\pi_*(\mathcal{O}_X(m(K_X+\Delta))\otimes(\mathcal{O}_X/\rho_*\mathcal{O}_Y(-Z)))
$$
 is surjective.
Note that $x$ is an isolated point on $\operatorname{Supp}(\mathcal{O}_X/\rho_*\mathcal{O}_Y(-Z))$.
In particular, 
$$
\pi^*\pi_*(\mathcal{O}_X(m(K_X+\Delta))\otimes(\mathcal{O}_X/\rho_*\mathcal{O}_Y(-Z)))\to
\mathcal{O}_X(m(K_X+\Delta))\otimes(\mathcal{O}_X/\rho_*\mathcal{O}_Y(-Z))
$$
is surjective at $x$.
Thus, considering the commutative diagram
$$
\xymatrix{
\pi^*\pi_*\mathcal{O}_X(m(K_X+\Delta)) \ar[r] \ar[d]&
\pi^*\pi_*(\mathcal{O}_X(m(K_X+\Delta))\otimes(\mathcal{O}_X/\rho_*\mathcal{O}_Y(-Z))) \ar[d]
\\
\mathcal{O}_X(m(K_X+\Delta))\ar[r]&
\mathcal{O}_X(m(K_X+\Delta))\otimes(\mathcal{O}_X/\rho_*\mathcal{O}_Y(-Z)),
}
$$
we conclude that $x\notin \operatorname{Bs}|m(K_X+\Delta)/S|$.
\end{proof}

\begin{rem}
\begin{itemize}
\item[(1)]
Suppose that $\dim S=0$ and $\operatorname{char}k>0$.
In applications of Corollary~\ref{cor: vanishing} in the proofs of Lemmas~\ref{lem: component of coeff 1} and \ref{lem: 0-dim non-klt}, one needs to verify
an inequality of the form
\[
\dim |L_m|
\geq
\dim_k\ker\left(
F^e\colon H^1(V,\mathcal O_V)\to H^1(V,\mathcal O_V)
\right)
\qquad\text{for }e\gg0,
\]
where $V$ is a fixed surface and $L_m$ is a divisor depending on $m$.
The right-hand side is bounded independently of $m$, since
$H^1(V,\mathcal O_V)$ is finite-dimensional. 
On the other hand, in our
situations, $L_m$ is obtained from
the birational pullback of $m(K_X+\Delta)$ by subtracting a divisor independent of $m$. 
Since
$K_X+\Delta$ is big, we have
\[
\dim |L_m| \sim O(m^2)\longrightarrow\infty
\qquad\text{as }m\to\infty.
\]
Thus, after taking $m\gg0$, the required inequality is automatically
satisfied.

\item[(2)]
When $\operatorname{char}k=0$ or $\dim S>0$,
the assumption $m\gg 0$ in Lemma~\ref{lem: component of coeff 1}~(1) and Lemma~\ref{lem: 0-dim non-klt} can be dropped.
Moreover, in Lemma~\ref{lem: component of coeff 1}~(2), the same conclusion holds under the weaker assumption $\pi_*(\mathcal{O}_X(K_X+D)|_C)\neq 0$.
\end{itemize}
\end{rem}

By Lemmas~\ref{lem: component of coeff 1} and \ref{lem: 0-dim non-klt}, we have shown that no non-klt center of $(X,\Delta)$ is contained in the base locus of $|m(K_X+\Delta)/S|$ for sufficiently large and divisible $m$.
We are now ready to prove Theorem~\ref{thm: abundance big case}:

\begin{proof}[Proof of Theorem~\ref{thm: abundance big case}]
In the relative case, by shrinking $S$ to an open neighborhood of each point if necessary, we may assume that $S$ is affine. 

Fix an integer $m \gg 0$, and suppose to the contrary that $\operatorname{Bs}|m(K_X+\Delta)/S| \neq \emptyset$. 
Take general members $\Xi_1, \Xi_2, \Xi_3 \in |m(K_X+\Delta)/S|$, and set 
\[
\Theta := \Xi_1 + \Xi_2 + \Xi_3.
\]
By Lemmas~\ref{lem: component of coeff 1} and \ref{lem: 0-dim non-klt}, we may assume that $\Theta$ does not contain any non-klt center of $(X, \Delta)$. 
On the other hand, $(X, \Delta + \Theta)$ is not lc %log canonical
at the generic points of $\operatorname{Bs}|m(K_X+\Delta)/S|$. 
Set
\[
c := \max \left\{ t \in \mathbb{R} \;\middle|\; (X, \Delta + t\Theta) \text{ is lc
outside the non-lc locus of } (X, \Delta) \right\}.
\]
Then $c \in \mathbb{Q}$ and $0 < c < 1$. Moreover, $\Delta + c\Theta$ is a boundary $\mathbb{Q}$-divisor satisfying
\[
K_X + \Delta + c\Theta \sim_{\mathbb{Q}, \pi} (1 + 3cm)(K_X + \Delta).
\]
Furthermore, since $\Xi_i$'s are general, there exists an lc %log canonical
center $C$ of $(X, \Delta + c\Theta)$ such that $C \subseteq \operatorname{Bs}|m(K_X+\Delta)/S|$.
Choose positive integers $l, n \in \mathbb{Z}_{>0}$ such that 
\[
l(K_X + \Delta + c\Theta) \sim_{\pi} nm(K_X+\Delta).
\]
By Lemmas \ref{lem: component of coeff 1} and \ref{lem: 0-dim non-klt}, we have
\[
C \nsubseteq \operatorname{Bs}|kl(K_X + \Delta + c\Theta)/S| \quad \text{for } k \gg 0.
\]
Therefore, we obtain the strict inclusion
\[
\operatorname{Bs}|knm(K_X+\Delta)/S| \subsetneq \operatorname{Bs}|m(K_X+\Delta)/S|
\]
since $C \nsubseteq \operatorname{Bs}|knm(K_X+\Delta)/S|$ while $C \subseteq \operatorname{Bs}|m(K_X+\Delta)/S|$.
By Noetherian induction, this strictly decreasing sequence of base loci implies that $\operatorname{Bs}|m(K_X+\Delta)/S| = \emptyset$ for all $m \gg 0$. Hence, $K_X+\Delta$ is $\pi$-semiample.
\end{proof}

As an immediate corollary of this theorem, we obtain the finite generation of the log canonical ring:

\begin{cor}
Let $\pi\colon X\to S$ be a proper surjective morphism from a normal surface $X$ to a variety $S$ over $k$.
Let $\Delta$ be a boundary $\mathbb{Q}$-divisor on $X$ such that $K_X+\Delta$ is $\mathbb{Q}$-Cartier and $\pi$-nef.
Then, the log canonical ring
\[
\bigoplus_{m \ge 0} \pi_* \mathcal{O}_X(\lfloor m(K_X + \Delta)\rfloor)
\]
is a locally finitely generated $\mathcal{O}_S$-algebra.
\end{cor}

\begin{proof}
This follows from Theorem~\ref{thm: abundance big case} and {\cite[Theorem~4.1]{Fuj84}}.
\end{proof}

\section{\texorpdfstring{Abundance for $\R$-divisors}{Abundance for R-divisors}} \label{sec: R-div}
In this section, we reduce the abundance theorem for $\R$-divisors to the case of $\Q$-divisors and complete the proof of the main theorem.

\begin{proof}[Proof of Theorem\ref{thm: abundance}]
  If $\Delta$ is a $\Q$-divisor, then the assertion follows from Theorems~\ref{thm: non-vanishing}, \ref{thm: log CY}, \ref{lem: abandance non-big} and \ref{thm: abundance big case}, together with {\cite[Theorem~4.1]{Fuj84}}.

  In general, by applying the Shokurov polytope argument, the case of $\mathbb{R}$-divisors can be reduced to that of $\mathbb{Q}$-divisors. 
For completeness, we briefly describe the relevant Shokurov polytope argument.
Let $V$ be the real vector space freely generated by the irreducible components of $\Delta$. 
Let $\mathcal{P}\subset V$ be any rational polytope consisting of $\R$-divisors $D$ in $V$ such that $D$ is boundary and $K_X+D$ is $\R$-Cartier.
Let $m>0$ be a positive number, and $\{C_{t}\}_{t\in T}$ be any set of $\pi$-exceptional curves on $X$ which satisfies $-(K_{X}+D)\cdot C_{t}\le m$ for each $D\in \mathcal{P}$.
Then, the subset
$$
\mathcal{N}:=\{D\in \mathcal{P}\ |\ \text{$(K_{X}+D)\cdot C_{t}\ge 0$ for any $t\in T$ } \}
$$
is a rational polytope.
Indeed, this can be proved by induction on $\dim\mathcal{P}$, using the same argument as in the proof of \cite[Proposition~3.2]{Bir11} (replacing a family of extremal rays by $\{C_{t}\}_{t\in T}$).
In our situation, we may take $\mathcal{P}$ to be the largest rational polytope consisting of boundary divisors $D\in V$ for which $K_X+D$ is $\mathbb{R}$-Cartier.
Let 
$$
m:=\max{\left\{3, -(K_{X}+D_{j})\cdot D_{j}, -K_{X}\cdot D_{j}\right\} },
$$
where $D_{j}$ runs through the irreducible components of $\Delta$,
and take $\{C_t\}_{t\in T}$ to be the set of all $\pi$-exceptional curves on $X$ 
that satisfies $-(K_{X}+D)\cdot C_{t}\le m$ for each $D\in \mathcal{P}$.

\begin{claim}
    For each $D\in \mathcal{P}$, the divisor $K_{X}+D$ is $\pi$-nef if and only if $(K_{X}+D)\cdot C_{t}\ge 0$ for any $t\in T$.
\end{claim}

\begin{proof}
It suffices to prove the converse implication.
If $\dim S\ge 1$, then $\{C_{t}\}_{t\in T}$ consists of all $\pi$-exceptional curves.
Indeed, any $\pi$-exceptional curve $C$ satisfies $C^{2}\le 0$.
Thus, the standard computation (see {\cite[Lemma~3.9]{Tan14}}) shows that $-(K_{X}+D)\cdot C\le 2$ for every $D\in \mathcal{P}$.
Thus, every $\pi$-exceptional curve $C$ belongs to $\{C_t\}_{t\in T}$, and the claim follows.

Hence we may assume that $\dim S=0$.
Assume that $-(K_{X}+D)\cdot C>0$ for some curve $C$.
Let $\rho\colon Y\to X$ be the minimal resolution of $X$ and let $D_Y$ be an effective divisor defined as $K_{Y}+D_{Y}=\rho^{*}(K_X+D)$.
By the Bend-and-Break theorem ({\cite[Theorem~1.13]{KoMo98}}, see also {\cite[Theorem~3.7]{Tan14}}), applied to the proper transform of $C$,
there exists a rational curve $Z$ on $Y$ such that $-K_{Y}\cdot Z\le 3$ and $-(K_{Y}+D_{Y})\cdot Z>0$.
Set $C':=\rho_{*}Z$.
Then, 
$$
-(K_{X}+D)\cdot C'=-(K_{Y}+D_{Y})\cdot Z>0.
$$
It remains to show that $C'=C_{t}$ for some $t\in T$.
To this end, let $D'\in \mathcal{P}$ be arbitrary, and define $D'_{Y}$ as  $K_{Y}+D'_{Y}=\rho^{*}(K_X+D')$.
If $C'=D_{j}$ for some $j$, then
$$
-(K_X+D')\cdot C'\le \max{\{ -(K_X+D_j)\cdot D_j, -K_{X}\cdot D_j\}}.
$$
If $C'\neq D_j$ for any $j$, then
$$
-(K_X+D')\cdot C'=-(K_{Y}+D'_{Y})\cdot Z\le -K_{Y}\cdot Z\le 3.
$$
In either case, the length bound in the definition of $\{C_t\}_{t\in T}$ shows that $C'=C_t$ for some $t\in T$.
\end{proof}

Thus, the resulting rational polytope $\mathcal{N}$ consists of all $D\in \mathcal{P}$ such that $K_{X}+D$ is $\pi$-nef, and then contains $\Delta$. Hence, writing $\Delta$ as a convex combination of rational points of $\mathcal{N}$, we obtain
$$
\Delta=\sum_i r_i\Delta_i,\quad r_i\ge 0,\quad \sum_i r_i=1,
$$
where each $\Delta_i$ is a boundary $\mathbb{Q}$-divisor and $K_X+\Delta_i$ is $\Q$-Cartier and $\pi$-nef. 
The abundance theorem for $\mathbb{Q}$-divisors therefore applies to each $\Delta_i$, which yields the desired statement for $\Delta$.
\end{proof}

\section{Examples} \label{sec: example}
In this section, we provide concrete examples showing that the assumptions in Theorem~\ref{thm: abundance} are natural.
In this section, suppose that the base field $k$ is not the algebraic closure of a finite field.

\begin{ex}
By {\cite[Example~A.3]{Tan14}} (resp.\  {\cite[Example~4.28]{Tan14}} or {\cite[Example~6.1]{Has25}}), for any $\varepsilon>0$, there exist a smooth projective surface $X$ and an effective $\Q$-divisor $\Delta$ on $X$ such that $(1-\varepsilon)\Delta$ is a boundary divisor, while $\Delta$ itself is not a boundary, and $K_{X}+\Delta$ is nef with $\kappa(X, K_{X}+\Delta)=2$ (resp.\ $\kappa(X, K_{X}+\Delta)=0$), but not semiample.
These examples show that the assumption that $\Delta$ is a boundary is essential in Theorems~\ref{thm: log CY} and \ref{thm: abundance big case}.
\end{ex}

\begin{ex} \label{ex: Mor26pre}
Recall the example constructed in \cite[Example~4.5]{Sak87}. Note that this example is also treated including the case of positive characteristic in \cite{Mor26pre}.
Let $p\colon Y_0=\mathbb{P}_{B}(\mathcal{E})\to B$ be a $\mathbb{P}^1$-bundle over an elliptic curve $B$ associated with the decomposable vector bundle $\mathcal{E}=\mathcal{O}_{B}\oplus \mathcal{O}_{B}(\mathfrak{e})$, where $\mathfrak{e}$ is a non-torsion divisor of degree $0$.
Let $T_1=T_{\mathcal{E}}$ and $T_2\sim T_{\mathcal{E}}-p^{*}\mathfrak{e}$ be the two sections of $p$.
We consider the composition of two blow-ups 
$$
\varphi\colon Y=Y_{2}\xrightarrow{\varphi_{2}}Y_{1}\xrightarrow{\varphi_{1}} Y_0
$$
defined as follows.
Let $\varphi_{1}$ be the blow-up at a point $P_1\in T_{1}$, and denote its exceptional curve by $E_1$.
Let $\varphi_{2}$ be the blow-up at a point $P_{2}\in E_{1}$ which does not lie on the proper transform of $T_{1}$.
We continue to denote by $T_1$ and $E_1$ their proper transforms on $Y$.
By {\cite[Lemmas~4.2 and 4.3]{Mor26pre}}, there exists a projective contraction $\rho\colon Y\to X$ onto a normal projective non-$\Q$-Gorentein surface $X$ which contracts precisely $T_1$ and $E_1$.
By the standard calculation of the Mumford pullback, we have 
$$
K_{Y}+2T_1+E_1=\rho^{*}K_{X}.
$$
On the other hand, 
$$
K_{Y}+2T_1+E_1=\varphi^{*}(K_{Y_0}+2T_1)\sim p^{*}\mathfrak{e}.
$$
Since $\mathfrak{e}$ is non-torsion,
 $K_{X}$ is numerically trivial but not $\Q$-linearly equivalent to any effective $\Q$-divisor.
Thus, the examples $(X, 0)$ and $(K_{Y_0}, 2T_1)$ show, respectively, that the $\Q$-Cartierness condition and the boundary condition in Theorem~\ref{thm: non-vanishing} cannot be omitted.

Moreover, by {\cite[Lemmas~4.4]{Mor26pre}}, the blow-ups $\varphi_{i}$ can be chosen so that there exist a normal projective $\Q$-Gorenstein surface $X'$ and a $K_{X'}$-negative extremal contraction $\psi\colon X'\to X$.
In particular, $K_{X'}$ is $\Q$-Cartier and pseudo-effective but not $\Q$-linearly equivalent to any effective $\Q$-divisor.
Thus, the non-vanishing theorem fails for log surfaces in full generality.
\end{ex}

\bibliographystyle{amsalpha}
\bibliography{ref}

\providecommand{\bysame}{\leavevmode\hbox to3em{\hrulefill}\thinspace}
\providecommand{\MR}{\relax\ifhmode\unskip\space\fi MR }
\providecommand{\MRhref}[2]{%
  \href{http://www.ams.org/mathscinet-getitem?mr=#1}{#2}
}
\providecommand{\href}[2]{#2}
\begin{thebibliography}{KMM94}

\bibitem[Ati57]{Ati57}
M.~F. Atiyah, \emph{Vector bundles over an elliptic curve}, Proc. London Math. Soc. (3) \textbf{7} (1957), 414--452. \MR{131423}

\bibitem[B{\u{a}}d01]{Bad01}
Lucian B{\u{a}}descu, \emph{Algebraic surfaces}, Universitext, Springer-Verlag, New York, 2001, Translated from the 1981 Romanian original by Vladimir Ma\c sek and revised by the author. \MR{1805816}

\bibitem[Bir11]{Bir11}
Caucher Birkar, \emph{On existence of log minimal models {II}}, J. Reine Angew. Math. \textbf{658} (2011), 99--113. \MR{2831514}

\bibitem[BLR90]{BLR90}
Siegfried Bosch, Werner L\"utkebohmert, and Michel Raynaud, \emph{N\'eron models}, Ergebnisse der Mathematik und ihrer Grenzgebiete (3) [Results in Mathematics and Related Areas (3)], vol.~21, Springer-Verlag, Berlin, 1990. \MR{1045822}

\bibitem[Eno23]{Eno23}
Makoto Enokizono, \emph{Vanishing theorems and adjoint linear systems on normal surfaces in positive characteristic}, Pacific J. Math. \textbf{324} (2023), no.~1, 71--110. \MR{4604667}

\bibitem[Eno24]{Eno24}
\bysame, \emph{An integral version of {Z}ariski decompositions on normal surfaces}, Eur. J. Math. \textbf{10} (2024), no.~2, Paper No. 38, 50. \MR{4754327}

\bibitem[FM92]{FoMc92}
Lung-Ying Fong and James M\textsuperscript{c}Kernan, \emph{Log abundance for surfaces}, Flips and abundance for algebraic threefolds (J{\'a}nos Koll{\'a}r, ed.), Ast{\'e}risque, no. 211, Soci{\'e}t{\'e} Math{\'e}matique de France, 1992, pp.~127--137.

\bibitem[Fuj84]{Fuj84}
Takao Fujita, \emph{Fractionally logarithmic canonical rings of algebraic surfaces}, J. Fac. Sci. Univ. Tokyo Sect. IA Math. \textbf{30} (1984), no.~3, 685--696. \MR{731524}

\bibitem[Fuj12]{Fuj12}
Osamu Fujino, \emph{Minimal model theory for log surfaces}, Publ. Res. Inst. Math. Sci. \textbf{48} (2012), no.~2, 339--371. \MR{2928144}

\bibitem[Fuj21]{Fuj21}
\bysame, \emph{On minimal model theory for algebraic log surfaces}, Taiwanese J. Math. \textbf{25} (2021), no.~3, 477--489. \MR{4298910}

\bibitem[GP96]{GP96}
Francisco~Javier Gallego and B.~P. Purnaprajna, \emph{Normal presentation on elliptic ruled surfaces}, J. Algebra \textbf{186} (1996), no.~2, 597--625. \MR{1423277}

\bibitem[Har77]{Har77}
Robin Hartshorne, \emph{Algebraic geometry}, Graduate Texts in Mathematics, vol. No. 52, Springer-Verlag, New York-Heidelberg, 1977. \MR{463157}

\bibitem[Has25]{Has25}
Kenta Hashizume, \emph{Minimal model program for normal pairs along log canonical locus}, Forum Math. Sigma \textbf{13} (2025), Paper No. e143, 71. \MR{4958721}

\bibitem[KM98]{KoMo98}
J\'anos Koll\'ar and Shigefumi Mori, \emph{Birational geometry of algebraic varieties}, Cambridge Tracts in Mathematics, vol. 134, Cambridge University Press, Cambridge, 1998, With the collaboration of C. H. Clemens and A. Corti, Translated from the 1998 Japanese original. \MR{1658959}

\bibitem[KMM94]{KMM94}
Sean Keel, Kenji Matsuki, and James M\textsuperscript{c}Kernan, \emph{Log abundance theorem for threefolds}, Duke Math. J. \textbf{75} (1994), no.~1, 99--119. \MR{1284817}

\bibitem[Kon10]{Kon10}
Kazuhiro Konno, \emph{Chain-connected component decomposition of curves on surfaces}, J. Math. Soc. Japan \textbf{62} (2010), no.~2, 467--486. \MR{2662852}

\bibitem[Mor26a]{Mor26pre}
Nao Moriyama, \emph{A note on $\mathbb{Q}$-{G}orenstein surfaces}, 2026, preprint.

\bibitem[Mor26b]{Mor26}
\bysame, \emph{Remarks on the minimal model theory for log surfaces in the analytic setting}, Nagoya Math. J. \textbf{261} (2026), Paper No. e27, 20. \MR{5057703}

\bibitem[Nak20]{Nak20}
Noboru Nakayama, \emph{On normal {Moishezon} surfaces admitting non-isomorphic surjective endomorphisms}, 2020, Preprint RIMS-1923, RIMS, Kyoto University.

\bibitem[Sak87]{Sak87}
Fumio Sakai, \emph{Classification of normal surfaces}, Algebraic geometry, {B}owdoin, 1985 ({B}runswick, {M}aine, 1985), Proc. Sympos. Pure Math., vol. 46, Part 1, Amer. Math. Soc., Providence, RI, 1987, pp.~451--465. \MR{927967}

\bibitem[Tan14]{Tan14}
Hiromu Tanaka, \emph{Minimal models and abundance for positive characteristic log surfaces}, Nagoya Math. J. \textbf{216} (2014), 1--70. \MR{3319838}

\end{thebibliography}

\end{document}